\documentclass[reqno]{amsart}
\usepackage[sc]{mathpazo}

\usepackage{amsmath}
\usepackage{amsthm}
\usepackage{amsfonts}
\usepackage{graphicx,psfrag,epsfig}
\usepackage{color}
\usepackage{mathrsfs}
\usepackage{pdfsync}
\usepackage{cite}

\newtheorem{thm}{Theorem}[section]
\newtheorem{lem}[thm]{Lemma}
\newtheorem{cor}[thm]{Corollary}
\newtheorem{prop}[thm]{Proposition}

\newtheorem{rem}[thm]{Remark}

\DeclareMathAlphabet{\mathpzc}{OT1}{pzc}{m}{it}

\numberwithin{equation}{section}

\newcommand{\bqn}{\begin{equation}}
\newcommand{\eqn}{\end{equation}}
\newcommand{\bqnn}{\begin{equation*}}
\newcommand{\eqnn}{\end{equation*}}

\newcommand{\R}{\mathbb{R}}

\newcommand{\B}{D}

\newcommand{\ve}{\varepsilon}
\newcommand{\e}{\varepsilon}
\newcommand{\rd}{\mathrm{d}}

\newcommand{\dhr}{\mathrel{\lhook\joinrel\relbar\kern-.8ex\joinrel\lhook\joinrel\rightarrow}}

\title[A free boundary problem for MEMS with linear bending]
{A free boundary problem modeling electrostatic MEMS:\\ I. Linear bending effects}

\author{Philippe Lauren\c{c}ot}
\address{Institut de Math\'ematiques de Toulouse, CNRS UMR~5219, Universit\'e de Toulouse \\ F--31062 Toulouse Cedex 9, France}
\email{laurenco@math.univ-toulouse.fr}

\author{Christoph Walker}
\address{Leibniz Universit\"at Hannover\\ Institut f\" ur Angewandte Mathematik \\ Welfengarten 1 \\ D--30167 Hannover\\ Germany}
\email{walker@ifam.uni-hannover.de}

\begin{document}

\date{\today}
\begin{abstract}
The dynamical and stationary behaviors of a fourth-order evolution equation with clam\-ped boundary conditions and a singular nonlocal reaction term, which is coupled to an elliptic free boundary problem on a non-smooth domain, are investigated.  The equation arises in the modeling of microelectromechanical systems (MEMS) and includes two positive parameters $\lambda$ and $\varepsilon$ related to the applied voltage and the aspect ratio of the device, respectively. Local and global well-posedness results are obtained for the corresponding hyperbolic and parabolic evolution problems as well as a criterion for global existence excluding the occurrence of finite time singularities  which are not physically relevant. Existence of a stable steady state is shown for sufficiently small~$\lambda$. Non-existence of steady states is also established when $\varepsilon$ is small enough and $\lambda$ is large enough (depending on $\varepsilon$). \end{abstract}


\maketitle

\section{Introduction}

Electrostatic actuators are typical microelectromechanical systems (MEMS), which consist of a conducting rigid ground plate above which an elastic membrane, coated with a thin layer of dielectric material and clamped on its boundary, is suspended, see Figure~\ref{MEMS1b}. Holding the ground plate and the deformable membrane at different electric potentials induces a Coulomb force across the device resulting in a deformation of the membrane and thus in a change in geometry. Mathematical models have been set up to predict the evolution of such MEMS in which the state of the device is fully described by the deformation $u$ of the membrane and the electrostatic potential $\psi$ in the device, see, e.g. \cite{LE08, PB03}. Assuming that there is no variation in the transverse horizontal direction and that the deformations are small, see Figure~\ref{fig1}, the evolution of $u=u(t,x)$ and $\psi=\psi(t,x,z)$ reads, after a suitable rescaling,
\begin{align}
\gamma^2\partial_t^2 u+\partial_t u +\beta\partial_x^4 u- \tau\partial_x^2 u&=-\lambda \big(\ve^2 \vert\partial_{x} \psi(x,u(x)) \vert^2 + \vert\partial_z \psi(x,u(x)) \vert^2\big)\ , & t>0\ ,& & x\in I\ , \label{hyper1}\\
u(t,\pm 1)= \beta \partial_x u(t,\pm 1)&=0\ ,& t>0\ ,& & \label{hyper2}\\
u(0,\cdot)=u^0\ ,\quad \gamma^2\partial_t u(0,\cdot)&= \gamma^2 u^1\ ,& && x\in I\ ,\label{hyper3}
\end{align}
where $I:=(-1,1)$. The right hand side of \eqref{hyper1} accounts for the electrostatic forces exerted on the membrane, where the parameter $\lambda>0$ is proportional to the square of the voltage difference  between the two components, and the parameter $\varepsilon>0$ denotes the aspect ratio (that is, the ratio height/length of the device). 
\begin{figure}
\centering\includegraphics[scale=.7]{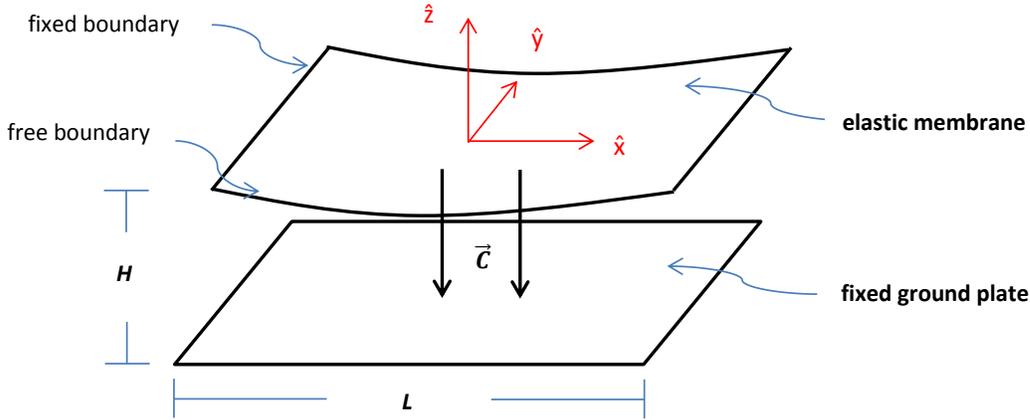}
\caption{\small Idealized electrostatic MEMS device.}\label{MEMS1b}
\end{figure}
The potential $\psi$ (suitably rescaled) satisfies
\begin{equation}\label{psi}
\ve^2 \,\partial_x^2\psi+\partial_z^2\psi=0 \ ,\quad (x,z)\in \Omega(u(t))\ ,\quad t>0\ ,
\end{equation}
in the time-dependent domain
$$
\Omega(u(t)):=\left\{(x,z)\,;\, x\in I\,,\, -1<z<u(t,x)\right\}\ ,
$$
between the ground plate and the membrane and is subject to the boundary conditions
\begin{equation}\label{psibc}
\psi(t,x,z)=\frac{1+z}{1+u(t,x)}\ ,\quad (x,z)\in\partial\Omega(u(t))\ ,\quad t>0\ .
\end{equation}
Recall that, in \eqref{hyper1}, $\gamma^2\partial_t^2 u$ and $\partial_t u$ account, respectively, for inertia and damping effects, while $\beta\partial_x^4 u$ and $-\tau\partial_x^2 u$ correspond to bending and stretching of the membrane, respectively. Thus, \eqref{hyper1} is a hyperbolic nonlocal semilinear equation for the membrane displacement $u$, which is coupled to the elliptic equation \eqref{psi} in the free domain $\Omega(u(t))$ for the electrostatic potential $\psi$. If damping effects dominate over inertia effects, one may neglect the latter by setting $\gamma=0$ and so obtains a parabolic equation for $u$. In this paper we shall investigate the hyperbolic problem as well as the parabolic one. 

Let us emphasize here that \eqref{hyper1}-\eqref{psibc} is meaningful only as long as the deformation $u$ stays above $-1$. From a physical viewpoint, when $u$ reaches the value $-1$ at some time $T_c>0$, that is, when
\begin{equation}
\lim_{t\to T_c} \min_{x\in I}\{ u(t,x) \} = -1\ , \label{touchdown}
\end{equation}
this corresponds to a touchdown of the deformable membrane on the ground plate, a phenomenon which has been observed experimentally in MEMS devices for sufficient large applied voltage values $\lambda$. In fact, the occurrence of this phenomenon is usually referred to as \textsl{pull-in instability} in physics literature and is characterized by the existence of a threshold value $\lambda_*$ for the applied voltage $\lambda$ with the following properties: touchdown occurs in finite time whenever $\lambda>\lambda_*$, but never takes place for $\lambda\in (0,\lambda_*)$. Obviously, the stable operating conditions of a given MEMS device heavily depend on the possible occurrence of this phenomenon, which may either be an expected feature of the device or irreversibly damage it. From this viewpoint, it is of great importance to test mathematical models for MEMS such as \eqref{hyper1}-\eqref{psibc} whether they exhibit such a touchdown behavior, that is, whether \eqref{touchdown} could occur. This question has been at the heart of a thorough mathematical analysis during the past decade for a simplified version of \eqref{hyper1}-\eqref{psibc}, the so-called \textsl{small aspect ratio model}. It is formally obtained from \eqref{hyper1}-\eqref{psibc} by setting $\varepsilon=0$ in \eqref{hyper1} and \eqref{psi}. In fact, setting $\varepsilon=0$ in \eqref{psi} and using \eqref{psibc} allows one to compute explicitly the electrostatic potential $\psi_0$ as a function of the yet to be determined deformation $u_0$ in the form
$$
\psi_0(t,x,z) = \frac{1+z}{1+u_0(t,x)}\ .
$$
The evolution equation for the deformation $u_0$ then reduces to
\begin{equation}
\gamma^2\partial_t^2 u_0+\partial_t u_0 +\beta\partial_x^4 u_0 - \tau\partial_x^2 u_0 = - \frac{\lambda}{(1+u_0)^2}\ , \qquad t>0\ , \quad x\in I\ , \label{sgm}
\end{equation}
supplemented with the clamped boundary conditions \eqref{hyper2} and the initial conditions \eqref{hyper3}. This approximation thus not only allows one to solve explicitly the free boundary value problem \eqref{psi}-\eqref{psibc}, but also reduces the nonlocal equation \eqref{hyper1} to a single semilinear equation with a still singular, but explicitly given reaction term. Furthermore, the right hand side of \eqref{sgm} is obviously monotone and concave with respect to $u_0$ and thus enjoys two highly welcome properties which are utmost helpful for the study of \eqref{sgm}: in particular, combined with the comparison principle, they yield the existence of the expected threshold value $\lambda_*$ of $\lambda$ such that there is no stationary solution for $\lambda>\lambda_*$ and at least one stable stationary solution for $\lambda\in (0,\lambda_*)$. The occurrence of the touchdown phenomenon \eqref{touchdown} for $\lambda>\lambda_*$ is also known to be true, but only for the second-order case $\beta=0$. We refer to  \cite{EspositoGhoussoubGuo, Gu10, GW09, KLNT11, LaurencotWalker_JAM, LinYang, LL12} and the references therein for a more complete description of the available results on the small aspect ratio model. We shall point out, however, that mainly the second-order case $\beta=0$ or the fourth-order case $\beta>0$ but with pinned boundary conditions $u=\beta\partial_x^2u=0$ (instead of the clamped boundary conditions \eqref{hyper2}) have been the focus of the mathematical research hitherto.

Unfortunately, the right hand side of \eqref{hyper1} does not seem to enjoy similar properties for $\varepsilon>0$ and so, we cannot rely on them to study the original free boundary problem \eqref{hyper1}-\eqref{psibc}. We thus shall take a different route in the spirit of the approach developed in \cite{ELW1, ELW2} for the second-order parabolic version of \eqref{hyper1}-\eqref{psibc} corresponding to the choice $\beta=\gamma=0$ of the parameters. Let us also mention that a quasilinear variant of the parabolic case $\gamma=0$ of \eqref{hyper1}-\eqref{psibc} with $\beta>0$ and curvature terms is investigated in the companion paper \cite{LW_QAM}, where the small deformation assumption is discarded.

\begin{figure}
\centering\includegraphics[width=10cm]{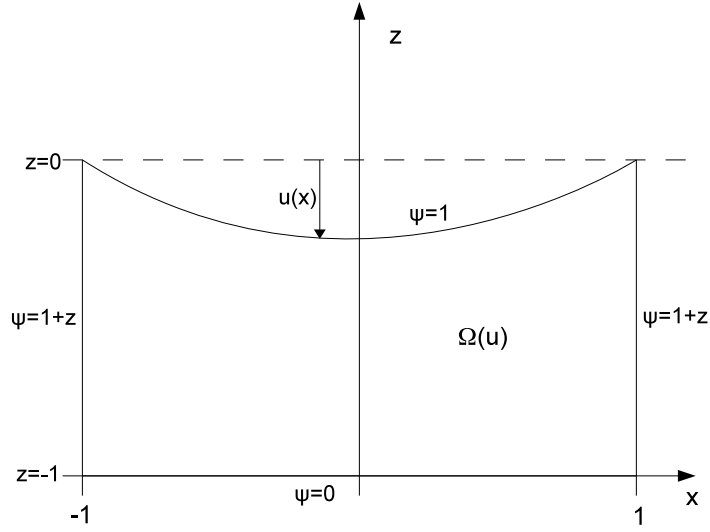}
\caption{\small A one-dimensional idealized electrostatic MEMS device.}\label{fig1}
\end{figure}

We focus in this paper on the case $\beta>0$, where bending is taken into account resulting in a fourth-order derivative in \eqref{hyper1} and an additional boundary condition in \eqref{hyper2} which has hardly been studied even for the small aspect ratio model \eqref{sgm} as mentioned above. Before describing more precisely the results of the analysis performed in this paper, let us first single out the main findings: the starting point is to establish the local well-posedness of \eqref{hyper1}-\eqref{psibc} along with an extension criterion guaranteeing global existence. As already observed in \cite{ELW1, ELW2}, the right hand side of \eqref{hyper1} is a nonlinear operator of, roughly speaking, order $3/2$ (in the sense that it maps $W_q^2(I)$ in $W_q^{\theta}(I)$ for all $\theta\in [0,1/2)$ and $q\in (2,\infty)$, see Proposition~\ref{L1} below). Since it also becomes singular when $u$ approaches $-1$, the extension criterion resulting from the fixed point argument leading to local well-posedness involves not only a lower bound on $u$, but also an upper bound on the norm of $(u,\gamma^2 \partial_t u)$ in a suitable Sobolev space. To be more precise, we first show that, if the maximal existence time $T_m$ of the solution $(u,\psi)$ to \eqref{hyper1}-\eqref{psibc} is finite, then
\begin{equation}\label{silbermond}
\limsup_{t\to T_m} \left( \|u(t)\|_{H^{2+2\alpha}(I)} + \gamma^2 \|\partial_t u(t)\|_{H^{2\alpha}(I)} \right) = \infty \;\;\text{ or }\;\; \liminf_{t\to T_m} \min_{[-1,1]} u(t) = -1
\end{equation}
for some $\alpha\in (0,1/4)$, see Proposition~\ref{Aglobal} and Proposition~\ref{ThyperIntroduction} below. The outcome of \eqref{silbermond} is not yet fully satisfactory from a physical point of view as it does not imply that finite time singularities are only due to the touchdown phenomenon \eqref{touchdown} described above. Precluding the occurrence of the finite time blowup of a Sobolev norm of $(u,\gamma^2 \partial_t u)$ requires more work and can subsequently be achieved by fully exploiting the additional information coming from the fourth-order derivative $\beta \partial_x^4 u$ as well as the underlying gradient flow structure of \eqref{hyper1}-\eqref{psibc}. Indeed, shape optimization computations reveal that \eqref{hyper1}-\eqref{psibc} may be seen as a gradient flow associated to the functional 
\begin{equation}
\mathcal{E}(u) := \mathcal{E}_b(u)+\mathcal{E}_s(u) - \lambda \mathcal{E}_e(u)\ , \label{TotalEnergy}
\end{equation}
which involves the mechanical energy $\mathcal{E}_b
+ \mathcal{E}_s$ given by
\begin{equation}
\mathcal{E}_b(u) := \frac{\beta}{2} \|\partial_x^2 u \|_{L_2(I)}^2\ , \quad \mathcal{E}_s(u) :=\frac{\tau}{2} \|\partial_x u \|_{L_2(I)}^2\ , \label{MechEnergy}
\end{equation}
and the electrostatic energy
\begin{equation}
\mathcal{E}_e(u) :=   \int_{\Omega(u)} \left[ \varepsilon^2 |\partial_x \psi_u(x,z)|^2 + |\partial_z \psi_u(x,z)|^2 \right]\ \mathrm{d}(x,z)\ , \label{ElecEnergy}
\end{equation}
the function $\psi_u$ denoting the solution to \eqref{psi}-\eqref{psibc} in $\Omega(u)$ for a given $u$. This fact seems to have been unnoticed up to now though it is inherent in the derivation of the model. Note, however, that the energy $\mathcal{E}$ is not coercive as it is the sum of three terms with different signs which do not seem to balance each other. Nevertheless, it plays an important r\^ole in our analysis since we show in Section~\ref{Sec6} that $\mathcal{E}_e(u)$ can be controlled by $\mathcal{E}_b(u)$ as long as $u$ stays bounded away from~$-1$,  provided that $\gamma$ is not too large. Recalling that $\mathcal{E}(u)$ is a decreasing function of time as a consequence of the gradient flow structure, such a control provides a bound on the $H^2(I)$-norm of $u$, still as long as touchdown does not occur. A bootstrap argument then implies that $(u,\gamma^2 \partial_t u)$ cannot blow up in that case and thus excludes that the finiteness of $T_m$ is due to the first statement in \eqref{silbermond}. Therefore,  when $\gamma$ is sufficiently small, we are able to prove a highly salient feature of the physical model: a finite time singularity is necessarily due to the touchdown phenomenon~\eqref{touchdown}. For large values of $\gamma$, it might be that oscillations created by the hyperbolic character of \eqref{hyper1} could interact with the touchdown phenomenon and give rise to more complicated dynamics. 

\medskip

We now state more precisely the main results. From now on the parameters $\varepsilon>0$, $\beta>0$, and $\tau\ge 0$ are fixed, additional restrictions on their ranges being made explicit in the statements of the results.

\subsection{Parabolic Case: $\gamma=0$}

We begin with the parabolic case $\gamma=0$ and first state its well-posedness along with a criterion for global existence which implies that a finite time singularity can only result from the touchdown phenomenon \eqref{touchdown}.

\begin{thm}[{\bf Well-Posedness}]\label{Alin}
Let $\gamma=0$. Consider an initial value $u^0\in  H^{4}(I)$ satisfying the boundary conditions $u^0(\pm 1)=\partial_x u^0(\pm 1)=0$ and such that $u^0(x)>-1$ for $x\in I$. Then, the following are true:

\begin{itemize}

\item[(i)] For each $\lambda>0$, there is a unique solution $(u,\psi)$ to \eqref{hyper1}-\eqref{psibc} on the maximal interval of existence $[0,T_m)$ in the sense that
$$
u\in C^1\big([0,T_m),L_2(I)\big)\cap  C\big([0,T_m), H^{4}(I)\big)
$$
satisfies \eqref{hyper1}-\eqref{hyper3} together with
$$
u(t,x)>-1\ ,\quad (t,x)\in [0,T_m)\times I\ , 
$$ 
and $\psi(t)\in H^2\big(\Omega(u(t))\big)$  solves \eqref{psi}-\eqref{psibc} in $\Omega(u(t))$ for each $t\in [0,T_m)$. 

\item[(ii)] If, for each $T>0$, there is $\kappa(T)\in (0,1)$ such that 
$$ 
u(t)\ge -1+\kappa(T)\ \text{ in } I\ \text{ for } t\in [0,T_m)\cap [0,T]\ ,
$$ 
then the solution exists globally in time, that is, $T_m=\infty$.

\item[(iii)] Given $\kappa\in (0,1)$, there are $\lambda_*(\kappa):=\lambda_*(\kappa,\ve)>0$ such that the solution exists globally in time provided that $\lambda\in (0,\lambda_*(\kappa))$ and $u^0\ge -1+\kappa$ in $I$ with $\|u^0\|_{H^{4}(I)}\le 1/\kappa$. Moreover,  $u\in L_\infty(0,\infty;H^{4}(I))$ in this case and 
$$
\inf_{(t,x)\in [0,\infty)\times I} u(t,x)>-1\ .
$$
\end{itemize}
\end{thm}

An important outcome of Theorem~\ref{Alin} is that the finiteness of $T_m$ corresponds to the occurrence of the touchdown phenomenon \eqref{touchdown} as stated in part (ii). This is in sharp contrast with the case $\beta=0$ studied in \cite{ELW1}, where the finiteness of $T_m$ could also be due to a blowup of the $W_q^2(I)$-norm of $u(t)$ as $t\to T_m$. The additional regularity of $u$ provided here by the fourth-order term $\beta \partial_x^4 u$ allows us to rule out the occurrence of this latter singularity. Also note that part~(iii) of Theorem~\ref{Alin} provides uniform estimates on the norm of $u$ and implies that touchdown does not even occur in infinite time.

\begin{rem}
Clearly, the maximal existence time $T_m>0$ depends not only on $\lambda>0$, but also on $\ve>0$.
\end{rem}

We perform the proof in Sections~\ref{Sec3a} and~\ref{Sec6}. We first use the regularizing properties of the parabolic operator $\partial_t + \beta \partial_x^4 - \tau \partial_x^2$ to set up a fixed point scheme and establish the local well-posedness of \eqref{hyper1}-\eqref{psibc} for all values of $\lambda$ and its global well-posedness for $\lambda$ sufficiently small. The results obtained are actually valid for less regular initial data, see Proposition~\ref{Aglobal} for a precise statement. A further outcome of this analysis is that solutions can be continued as long as $u$ stays above $-1$ and a suitable Sobolev norm of $u$ is controlled, as already outlined in \eqref{silbermond}. We subsequently show in Proposition~\ref{pr.rc1} that the former implies the latter, leading to Theorem~\ref{Alin}~(ii). An important step in the proof is the following energy equality (recall that $\mathcal{E}$, $\mathcal{E}_b$, $\mathcal{E}_s$, and $\mathcal{E}_e$ are defined in \eqref{TotalEnergy}-\eqref{ElecEnergy}). 

\begin{prop}\label{le.rc2}
Under the assumptions of Theorem~\ref{Alin}, 
\begin{equation}
\mathcal{E}(u(t)) + \int_0^t \|\partial_t u(s)\|_{L_2(I)}^2\ \mathrm{d}s = \mathcal{E}(u^0)\ ,\quad t\in [0,T_m)\ . \label{rc2}
\end{equation}
\end{prop}

The main difficulty in the proof of Proposition~\ref{le.rc2} is the computation of the derivative of $\mathcal{E}_e(u)$ with respect to $u$. Indeed, the dependence of $\mathcal{E}_e(u)$ on $u$ is somehow implicit and involves the domain $\Omega(u)$. Nevertheless, the derivative of $\mathcal{E}_e(u)$ with respect to $u$ can be interpreted as the shape derivative of the Dirichlet integral of $\psi_u$, which can be computed and shown to be equal to the right hand side of \eqref{hyper1} -- except for the sign -- by shape optimization arguments \cite{HP05}. Let us, however, mention that the time regularity of $u$ is not sufficient to apply directly the results in \cite{HP05} and an approximation has to be used, see Proposition~\ref{pr.z1} below.

\subsection{Hyperbolic Case: $\gamma>0$}

We next turn to the hyperbolic case $\gamma>0$ and show that results similar to Theorem~\ref{Alin} and Proposition~\ref{le.rc2} are available in that case as well, with two noticeable peculiarities: on the one hand, the lack of regularizing effects for the beam equation requires more regularity on the initial data. On the other hand, the extension of Theorem~\ref{Alin}~(ii) only seems possible for small values of $\gamma$, see Theorem~\ref{ThyperIntro.b}.

\begin{thm}\label{ThyperIntro}
Let $\gamma>0$. Consider an initial value $(u^0,u^1)\in H^{5}(I)\times H^{3}(I)$  satisfying $u^0(\pm 1) = \partial_x u^0(\pm 1) = u^1(\pm 1) = \partial_x u^1(\pm 1) = 0$ and such that $u^0> -1$ in~$I$. Then the following hold:

\begin{itemize}
\item[(i)] For each $\lambda>0$, there is a unique solution $(u,\psi)$ to \eqref{hyper1}-\eqref{psibc} on the maximal interval of existence $[0,T_m)$ in the sense that
$$
u\in C([0,T_m),H^{2}(I))\cap C^1([0,T_m),L_2(I))\ ,\quad \partial_t^k u\in L_1(0,T; H^{4-2k}(I))
$$
for $k=0,1,2$ and $T\in (0,T_m)$, and satisfies \eqref{hyper1}-\eqref{hyper3} together with
$$
u(t,x)>-1\ ,\quad (t,x)\in [0,T_m)\times I\ , 
$$ 
while $\psi(t)\in H^2\big(\Omega(u(t))\big)$  solves \eqref{psi}-\eqref{psibc}  in $\Omega(u(t))$ for each $t\in [0,T_m)$.

\item[(ii)] If, for each $T>0$, there is $\kappa(T)\in (0,1)$ such that 
$$
\|u(t)\|_{H^{3}(I)}+\|\partial_t u(t)\|_{H^{1}(I)}\le \frac{1}{\kappa(T)}\quad \text{and}\quad   u(t)\ge -1+\kappa(T)\ \text{ in } I
$$ for $t\in [0,T_m)\cap [0,T]$, then the solution exists globally in time, that is, $T_m=\infty$.

\item[(iii)] Given $\kappa\in (0,1)$, there are $\lambda(\kappa)>0$ and $N(\kappa)>0$ such that $T_m=\infty$ provided that $\lambda\le \lambda(\kappa)$ and $u^0\ge -1+\kappa$ in $I$ with 
$$ 
\|(u^0,u^1)\|_{H^{5}(I)\times H^{3}(I)}\le N(\kappa)\ .
$$
Moreover, $u\in L_\infty(0,\infty;H^{2}(I))$ in this case and
$$
\inf_{(t,x)\in [0,\infty)\times I} u(t,x)>-1\ .
$$
\end{itemize}
\end{thm}

The solution we construct is actually more regular under less regularity assumptions on the initial data, see Proposition~\ref{ThyperIntroduction} for a more precise statement. Next, if $\gamma$ is sufficiently small, we can prove, as in the parabolic case $\gamma=0$, that only the touchdown phenomenon \eqref{touchdown} may generate a finite time singularity. 

\begin{thm}\label{ThyperIntro.b}
There is $\gamma_0>0$ such that, if $\gamma\in (0,\gamma_0)$ and the initial value $(u^0,u^1)$ satisfies the assumptions of Theorem~\ref{ThyperIntro}, then the solution $(u,\psi)$ to \eqref{hyper1}-\eqref{psibc} obeys the following criterion for global existence: if for each $T>0$ there is $\kappa(T)\in (0,1)$ such that 
$$
u(t)\ge -1+\kappa(T)\ \text{ on } I
$$ 
for $t\in [0,T_m)\cap [0,T]$, then $T_m=\infty$.
\end{thm}

The starting point of the proofs of Theorem~\ref{Alin}~(ii) and Theorem~\ref{ThyperIntro.b} is to derive an upper bound for $u$, which does not depend on $T_m$. This is obvious when $\gamma=\beta=0$ as the non-positivity of the right hand side of \eqref{hyper1} and the comparison principle guarantee that $u(t)\le \|u_0\|_\infty$ for $t\in [0,T_m)$. This is no longer true when $\beta>0$, and we instead derive a weighted $L_1$-estimate for~$u$, still using the non-positivity of the right hand side of \eqref{hyper1}. This seems to require damping to dominate over inertia effects and thus that $\gamma$ is sufficiently small. Otherwise, this estimate might fail to be true due to the oscillatory behavior of the beam equation, which could propagate large (negative) values of the right hand side of \eqref{hyper1} to large (positive) values of $u$. 

Finally, as in the parabolic case, we have an energy equality:

\begin{prop}\label{le.rc2hyper}
Under the assumptions of Theorem~\ref{ThyperIntro}, 
\begin{equation}
\mathcal{E}(u(t)) + \frac{\gamma^2}{2}\| \partial_t u(t)\|_{L_2(I)}^2+\int_0^t \|\partial_t u(s)\|_{L_2(I)}^2\ \mathrm{d}s = \mathcal{E}(u^0)+\frac{\gamma^2}{2}\| u^1\|_{L_2(I)}^2\ ,\quad t\in [0,T_m)\ . \label{rc2b}
\end{equation}
\end{prop}

We shall point out that, on physical grounds, the maximal existence time $T_m$ is expected to be finite for large values of~$\lambda$. In this direction, let us recall that a classical technique to investigate the possible occurrence of finite time singularities is the so-called eigenfunction technique. Owing to the nonlocal character of the right hand side of \eqref{hyper1}, this technique does not seem to be appropriate here, but a nonlinear variant thereof introduced in \cite{ELW1} has proven  to be successful and allowed us to show the finiteness of $T_m$ for sufficiently large $\lambda$ in the second order parabolic case, that is, when $\gamma=\beta=0$. We have yet been unable to develop it further to achieve a similar result when $(\gamma,\beta)\ne (0,0)$, in particular when $\beta>0$, the fourth-order problem under investigation herein. The main difficulties are, on the one hand, that the comparison principle is no longer valid and there is no \textit{ a priori} upper bound on $u$. On the other hand, there are terms resulting from integration by parts involving the fourth-order derivative $\beta \partial_x^4 u$, which cannot be controlled in a suitable way.

However, a modification of the technique introduced in \cite{ELW1} proves to be useful for the stationary problem with $\beta>0$, leading us to a non-existence result for large values of $\lambda$ as explained in the following subsection. 

\subsection{Steady States}

We next consider time independent solutions and show that, as expected from physics, such solutions exist for $\lambda$ sufficiently small and do not exist for $\lambda$ large, the latter being true provided $\varepsilon$ is small.

\begin{thm}[{\bf Steady State Solutions}]\label{TStable2}
\begin{itemize}
\item[(i)] There is $\lambda_s>0$ such that for each $\lambda\in (0,\lambda_s)$ there exists an asymptotically stable steady state $(U_\lambda,\Psi_\lambda)$ to \eqref{hyper1}-\eqref{psibc} with $U_\lambda\in H^4(I)$ satisfying $-1<U_\lambda< 0$ in $I$ and  $\Psi_\lambda\in H^2(\Omega(U_\lambda))$.
\item[(ii)] There are $\ve_*>0$ and ${\lambda_c}: (0,\ve_*)\to (0,\infty)$ such that there is no steady state $(u,\psi)$ to \eqref{hyper1}-\eqref{psibc} for $\ve\in (0,\ve_*)$ and $\lambda>{\lambda_c}(\ve)$.
\end{itemize}
\end{thm}

We postpone a more precise statement and its proof to Section~\ref{Sec4}. Let us just mention that the existence of steady states for small values of $\lambda$ along with their asymptotic stability follows from the implicit function theorem and the principle of linearized stability, respectively. The non-existence is proved by a nonlinear variant of the eigenfunction method mentioned above. In this direction, we recall that a salient feature of the operator $\beta \partial_x^4 - \tau \partial_x^2$ in $H_D^4(I)$ is that it has a positive eigenfunction associated to its positive principal eigenvalue \cite{Gr02,LaurencotWalker_JAM,Ow97}. 

\section{Auxiliary Results}\label{Sec2a} 

In order to state precisely our existence results, we first introduce the (subspaces of) Bessel potential spaces $H_\B^{4\theta}(I)$ including clamped boundary conditions, if meaningful, by setting
$$
H_\B^{4\theta}(I):=\left\{\begin{array}{lll}
& \{v\in H^{4\theta}(I)\,;\, v(\pm 1)=\partial_x v(\pm 1)=0\}\ , & 4\theta>\dfrac{3}{2}\ ,\\
& \{v\in H^{4\theta}(I)\,;\, v(\pm 1)=0\}\ , & \dfrac{1}{2}<4\theta<\dfrac{3}{2}\ ,\\
&   H^{4\theta}(I)\ , & 4\theta<\dfrac{1}{2}\ .
\end{array}
\right.
$$
Note that the spaces $H_{\B}^{4\theta}(I)$  coincide with the complex interpolation spaces	
\bqn\label{interpol}
H_\B^{4\theta}(I) =\big[ L_2(I), H_\B^{4}(I)\big]_\theta\ ,\quad \theta\in [0,1]\setminus\left\{\frac{1}{8} , \frac{3}{8}\right\}\ ,
\eqn
except for equivalent norms, see \cite[Theorem 4.3.3]{Triebel}.

\medskip

We shall first recall properties of solutions to the Laplace equation \eqref{psi}-\eqref{psibc} in dependence of a given (free) boundary described by a function $u(t):[-1,1]\rightarrow (-1,\infty)$ for a fixed time~$t$. For that purpose we transform the free boundary problem \eqref{psi}-\eqref{psibc} to the fixed rectangle \mbox{$\Omega:=I\times (0,1)$}. More precisely, for a sufficiently smooth function $v:[-1,1]\rightarrow (-1,\infty)$ with $v(\pm 1)=0$, we define a diffeomorphism \mbox{$T_v:\overline{\Omega(v)}\rightarrow \overline{\Omega}$} by setting
\begin{equation}\label{Tu}
T_v(x,z):=\left(x,\frac{1+z}{1+v(x)}\right)\ ,\quad (x,z)\in \overline{\Omega(v)}\ ,
\end{equation}
with $\Omega(v) = \left\{ (x,z)\in I\times (-1,\infty)\ ;\quad -1 < z < v(x) \right\}$. Clearly, its inverse is
\begin{equation}\label{Tuu}
T_v^{-1}(x,\eta)=\big(x,(1+v(x))\eta-1\big)\ ,\quad (x,\eta)\in \overline{\Omega}\ ,
\end{equation}
and the Laplace operator $\ve^2\partial_x^2+\partial_z^2$ is transformed to the $v$-dependent differential operator $\mathcal{L}_v$ given by
\begin{equation}
\begin{split}
\mathcal{L}_v w\, :=\, & \e^2\ \partial_x^2 w - 2\e^2\ \eta\ \frac{\partial_x v(x)}{1+v(x)}\ \partial_x\partial_\eta w
+ \frac{1+\e^2\eta^2(\partial_x v(x))^2}{(1+v(x))^2}\ \partial_\eta^2 w\\
& + \e^2\ \eta\ \left[ 2\ \left(\frac{\partial_x v(x)}{1+v(x)} \right)^2 - \frac{\partial_x^2 v(x)}{1+v(x)} \right]\ \partial_\eta w\ .
\end{split} \label{acdc}
\end{equation}
Next, defining for $4\theta>2$ and $\kappa\in (0,1)$ the open subset
\begin{equation}
S_{\theta}(\kappa):=\left\{v\in H_\B^{4\theta}(I)\,;\, \|v\|_{H_\B^{4\theta}(I)}< 1/\kappa \;\;\text{ and }\;\; -1+\kappa< v(x) \text{ for } x\in I \right\} \label{setST}
\end{equation}
of $H_\B^{4\theta}(I)$ with closure 
$$
\overline{S}_{\theta}(\kappa)=\left\{v\in H_\B^{4\theta}(I)\,;\, \|v\|_{H_\B^{4\theta}(I)}\le 1/\kappa \;\;\text{ and }\;\; -1+\kappa\le v(x) \text{ for } x\in I \right\}\ ,
$$ 
we first collect crucial properties of the solution $\phi=\phi_v$ to the elliptic boundary value problem
\begin{eqnarray}
\big(\mathcal{L}_v \phi\big) (x,\eta)\!\!\!&=0\ ,&(x,\eta)\in\Omega\ ,\label{230}\\
\phi(x,\eta)\!\!\!&=\eta\ , &(x,\eta)\in \partial\Omega\ ,\label{240}
\end{eqnarray}
in dependence of a given $v\in\overline{S}_{\theta}(\kappa)$:

\begin{prop}\label{L1}
Let $4\theta>2$ and $\kappa\in (0,1)$. Then, for each $v\in\overline{S}_{\theta}(\kappa)$ there is a unique solution $\phi=\phi_v\in H^2(\Omega)$ to \eqref{230}-\eqref{240},
and there is a constant $c(\kappa)>0$ such that
\begin{equation*}
\|\phi_{v_1}-\phi_{v_2}\|_{H^2(\Omega)} \le c(\kappa)\, \|v_1-v_2\|_{H^{4\theta}(I)}\ ,\quad  v_1,\ v_2\in\overline{S}_{\theta}(\kappa)\ .
\end{equation*}
Moreover, for $4\sigma\in [0,1/2)$, the mapping 
$$
g: S_{\theta}(\kappa)\longrightarrow H_\B^{4\sigma}(I)\ ,\quad v\longmapsto \frac{1+\e^2(\partial_x v)^2}{(1+v)^2}\  \vert\partial_\eta\phi_v(\cdot,1)\vert^2
$$
is analytic, bounded, and uniformly Lipschitz continuous. 
\end{prop}

\begin{proof}
This follows from \cite[Proposition~2.1 \& Equation~(38)]{ELW1} by noticing that $H_\B^{4\theta}(I)\hookrightarrow W_q^2(I)$, with $q\in (2,\infty)$ chosen such that $4\theta>5/2-1/q$.
\end{proof}

Let now $u$ be a time-dependent function with $u(t)\in H_\B^{4\theta}(I)$ and  $-1+\kappa< u(t,x)$ for $x\in I $ and $t\ge 0$. Then, with the notation above, a function $\psi(t)=\psi_{u(t)}$ solves the boundary value problem \eqref{psi}-\eqref{psibc} if and only if $\phi(t)=\phi_{u(t)}=\psi(t)\circ T_{u(t)}^{-1}$ solves
\begin{eqnarray}
 \big(\mathcal{L}_{u(t)}\phi(t)\big) (x,\eta)\!\!\!&=0\ ,&(x,\eta)\in\Omega\ , \quad t>0\ ,\label{23}\\
\phi(t,x,\eta)\!\!\!&=\eta\ , &(x,\eta)\in \partial\Omega\ , \quad t>0\ .\label{24}
\end{eqnarray}
Observe that $\psi(t)=\psi_{u(t)}\in H^2(\Omega(u(t)))$ by Proposition~\ref{L1} and that regarding the right hand side of equation~\eqref{hyper1} we have the relation
$$
-\lambda  \big(\ve^2 \vert\partial_{x} \psi(t,\cdot,u(t,\cdot)) \vert^2 + \vert\partial_z \psi(t,\cdot,u(t,\cdot)) \vert^2\big)  = -\lambda\, \frac{1+\e^2(\partial_x u(t,\cdot))^2}{(1+u(t,\cdot))^2} \ \vert\partial_\eta\phi(t,\cdot,1)\vert^2 =-\lambda\, g(u(t))\ ,
$$
since $\partial_x\phi(t,x,1)=0$ for $x\in I$ due to $\phi(t,x,1)=1$ by \eqref{24}. Let us point out that Proposition~\ref{L1} and the just introduced notation put us in a position to formulate \eqref{hyper1}-\eqref{psibc} as a  single nonlocal evolution equation only involving the deflection $u$, see \eqref{CPP} below.

\medskip

We next prepare the proof of the energy identities \eqref{rc2} and \eqref{rc2b}, which will be given in Section~\ref{SectEnergy}. Owing to the dependence of the electrostatic energy $\mathcal{E}_e$ on the domain $\Omega(u)$, it turns out that the time regularity of the $u$-component of the solution to \eqref{hyper1}-\eqref{psibc} given by Theorem~\ref{Alin} is not sufficient to proceed directly. We shall thus use an approximation argument based on the following result, the proof being inspired by techniques from shape optimization \cite{HP05}:

\begin{prop}\label{pr.z1}
Let $T>0$, $\nu>0$, and $\bar{u}\in C([0,T],H_{D}^{2+2 \nu}(I))\cap C^1([0,T],H_{D}^{1+2\nu}(I))$ be such that $\bar{u}(t,x)>-1$ for $(t,x)\in [0,T]\times [-1,1]$. Then
\begin{equation}
\mathcal{E}_e(\bar{u}(t_2)) - \mathcal{E}_e(\bar{u}(t_1)) = - \int_{t_1}^{t_2} \int_{-1}^1 g(\bar{u}(s))\, \partial_t \bar{u}(s)\ \rd x \mathrm{d}s\ , \quad 0\le t_1 \le t_2 \le T\ . \label{z1}
\end{equation}
\end{prop}

\begin{proof}
We fix $q>2$ such that the embedding of $H^{2+2\nu}(I)$ in $W_q^2(I)$ and that of $H^{1+2\nu}(I)$ in $W_q^1(I)$ are continuous. To simplify notation, we let, for each $t\in [0,T]$, $\phi(t)=\phi_{\bar{u}(t)}\in H^2(\Omega)$ be the solution to \eqref{23}-\eqref{24} associated to $\bar{u}(t)$  as provided by Proposition~\ref{L1} and $\psi(t)=\psi_{\bar{u}(t)}\in H^2(\Omega(\bar{u}(t)))$ be the corresponding solution to \eqref{psi}-\eqref{psibc} also associated to $\bar{u}(t)$. Recall that the electrostatic energy $\mathcal{E}_e(\bar{u})$ is given by
$$
\mathcal{E}_e(\bar{u}(t)) = \int_{\Omega(\bar{u}(t))} \left( \varepsilon^2 |\partial_x \psi(t,x,z) |^2 + |\partial_z\psi(t,x,z)|^2 \right)\ \mathrm{d}(x,z)\ ,
$$
where $\Omega(\bar{u}(t)) = \{ (x,z) \in I\times\mathbb{R}\ : \ -1 < z < \bar{u}(t,x) \}$, and set
$$
U(t,x) := \frac{\partial_x \bar{u}(t,x)}{1+\bar{u}(t,x)}  \;\;\text{ and }\;\; \Phi(t,x,\eta) := \phi(t,x,\eta)-\eta\ , \qquad (t,x,\eta)\in [0,T]\times\Omega\ .
$$ 

\medskip

\noindent\textbf{Step~1: Alternative formula for $\mathcal{E}_e(\bar{u})$}.  Let $t\in [0,T]$. Since $$\psi(t,x,z) = \phi\left(t,x,\frac{1+z}{1+\bar{u}(t,x)}\right)\ ,\quad (x,z)\in \Omega(\bar{u}(t))\ ,$$ a simple change of variables reveals that 
\begin{align}
\mathcal{E}_e(\bar{u}(t)) = & \varepsilon^2 \int_\Omega \left| \partial_x \phi(t,x,\eta) - \eta U(t,x) \partial_\eta \phi(t,x,\eta) \right|^2 (1+\bar{u}(t,x))\ \mathrm{d}(x,\eta) \nonumber \\
& + \int_\Omega \frac{|\partial_\eta \phi(t,x,\eta)|^2}{1+\bar{u}(t,x)}\ \mathrm{d}(x,\eta)\ . \label{z4}
\end{align} 

\medskip

\noindent\textbf{Step~2: Time differentiability of $\Phi$ and $\phi$}. 
Recall that, for $t\in [0,T]$, $\Phi(t)$ solves
\begin{equation}
\mathcal{L}_{\bar{u}(t)} \Phi(t) = f(t) \;\;\text{ in }\;\;\Omega\ , \qquad \Phi(t) = 0 \;\;\text{ on }\;\;\partial\Omega\ , \label{z2} \\
\end{equation}
with
$$
f(t,x,\eta) := \varepsilon^2 \eta \left[ \partial_x U(t,x) - |U(t,x)|^2 \right]\ , \qquad (t,x,\eta) \in [0,T]\times \Omega\ .
$$
For further use, we write the operator $\mathcal{L}_{\bar{u}(t)}$ in divergence form,
\begin{align}
\mathcal{L}_{\bar{u}(t)} w = & \partial_x \left( a_{11}(t) \partial_x w + a_{12}(t) \partial_\eta w \right) + \partial_\eta \left( a_{21}(t) \partial_x w + a_{22}(t) \partial_\eta w \right) \nonumber \\
& + b_1(t) \partial_x w + b_2(t) \partial_\eta w\ , \label{Ldiv}
\end{align}
and the function $f(t)$ in the form
$$
f(t) = \partial_x f_1(t) + f_2(t)\ ,
$$
where
\begin{align*}
& a_{11}(t,x,\eta) := \varepsilon^2\ ,  && a_{22}(t,x,\eta) := \frac{1}{(1+\bar{u}(t,x))^2} + \varepsilon^2 \eta^2 |U(t,x)|^2\ , \\
& a_{12}(t,x,\eta) = a_{21}(t,x,\eta) := - \varepsilon^2 \eta U(t,x) \ , \\
& b_1(t,x,\eta) := \varepsilon^2 U(t,x)\ , &&  b_2(t,x,\eta) := - \varepsilon^2 \eta |U(t,x)|^2\ , \\
& f_1(t,x,\eta) := \varepsilon^2 \eta U(t,x)\ , &&  f_2(t,x,\eta) := - \varepsilon^2 \eta |U(t,x)|^2
\end{align*}
for $(t,x,\eta) \in [0,T]\times\Omega$.
Now, for $t\in [0,T]$, $s\in [-t,T-t]$, and $(x,\eta)\in\Omega$, we define
$$
\delta_s \Phi(t,x,\eta) := \Phi(t+s,x,\eta) - \Phi(t,x,\eta) = \phi(t+s,x,\eta) - \phi(t,x,\eta)\ .
$$
Let $t\in [0,T]$. We readily deduce from \eqref{z2} (applied with $t$ and $t+s$) that
\begin{equation}
\mathcal{L}_{\bar{u}(t)} \delta_s\Phi(t) = R_s(t)  \;\;\text{ in }\;\;\Omega\ , \qquad \delta_s\Phi(t) = 0 \;\;\text{ on }\;\;\partial\Omega\ , \label{z3} 
\end{equation}
where
\begin{align*}
R_s(t) := & -\partial_x \left[ \delta_s a_{12}(t) \partial_\eta \Phi(t+s) \right] - \partial_\eta \left[ \delta_s a_{21}(t) \partial_x \Phi(t+s) \right] \\
& - \partial_\eta \left[ \delta_s a_{22}(t) \partial_\eta\Phi(t+s) \right] - \delta_s b_1(t) \partial_x \Phi(t+s) - \delta_s b_2(t) \partial_\eta \Phi(t+s) \\
& + \partial_x \delta_s  f_1(t) + \delta_s f_2(t)\ .
\end{align*}
Next, for $w\in W_{2,D}^1(\Omega):=\{w\in W_{2}^1(\Omega)\,:\, w=0\text{ on } \partial\Omega\}$, we use Green's formula to obtain
\begin{align*}
\int_\Omega R_s(t) w\ \mathrm{d}(x,\eta) = & \int_\Omega \delta_s a_{21}(t) \left( \partial_\eta \Phi(t+s) \partial_x w + \partial_x \Phi(t+s) \partial_\eta w \right)\ \mathrm{d}(x,\eta) \\
& +\int_\Omega \delta_s a_{22}(t) \partial_\eta \Phi(t+s) \partial_\eta w\ \mathrm{d}(x,\eta) - \int_\Omega \delta_s f_{1}(t) \partial_x w\ \mathrm{d}(x,\eta)\\
& - \int_\Omega \left( \delta_s b_1(t) \partial_x \Phi(t+s) + \delta_s b_2(t) \partial_\eta \Phi(t+s)b-\delta_s f_2(t) \right) w\ \mathrm{d}(x,\eta)\ .
\end{align*}
We now aim at investigating the behavior of $R_s(t)$ as $s\to 0$. To this end, we note that the regularity of $\Phi$ and the continuous embedding of $W_{2,D}^1(\Omega)$ in $L_{2q/(q-2)}(\Omega)$ guarantee that 
\begin{align}
& s\mapsto \partial_\eta \Phi(t+s) \partial_x w\ ,\ \partial_x \Phi(t+s) \partial_\eta w\ ,\ \partial_\eta \Phi(t+s) \partial_\eta w \in C([-t,T-t],L_{q/(q-1)}(\Omega)) \ ,\label{z6} \\
& s\mapsto \partial_x \Phi(t+s) w\ , \partial_\eta \Phi(t+s) w\ \in C([ -t,T-t],L_{q/(q-1)}(\Omega))\ , \label{z7}
\end{align}
while the regularity of $\bar{u}$ and the continuous embedding of $W_q^2(I)$ in $W_\infty^1(I)$ imply that 
\begin{equation}
a_{21}, a_{22}, b_1, b_2, f_1, \text{ and } f_2 \text{ belong to } C^1([0,T],L_q(I))\ . \label{z5}
\end{equation}
We may therefore define $R_0\in C([0,T],W_{2,D}^{-1}(\Omega))$ by
\begin{align*}
\langle R_0(t) , w \rangle :=  & \int_\Omega \partial_t a_{21}(t) \left( \partial_\eta \Phi(t) \partial_x w + \partial_x \Phi(t) \partial_\eta w \right)\ \mathrm{d}(x,\eta) \\
& + \int_\Omega \partial_t a_{22}(t) \partial_\eta \Phi(t) \partial_\eta w\ \mathrm{d}(x,\eta) - \int_\Omega \partial_t f_{1}(t) \partial_x w\ \mathrm{d}(x,\eta)\\
& - \int_\Omega \left( \partial_t b_1(t) \partial_x  \Phi(t) + \partial_t b_2(t) \partial_\eta \Phi(t) - \partial_t f_2(t) \right) w\ \mathrm{d}(x,\eta)
\end{align*}
for $t\in [0,T]$ and deduce from the regularity properties \eqref{z6}, \eqref{z7}, and \eqref{z5} that
\begin{equation}
\lim_{s\to 0} \sup_{t\in [0,T]}\left\{ \left\| \frac{R_s(t)}{s} - R_0(t) \right\|_{W_{2,D}^{-1}(\Omega)} \right\} = 0\ . \label{z8}
\end{equation}
Now, for each $t\in [0,T]$, it follows from \cite[Lemma~2.2]{ELW2} that there is a unique solution $I_0(t)\in W_{2,D}^1(\Omega)$ to
\begin{equation}
\mathcal{L}_{\bar{u}(t)} I_0(t) = R_0(t) \;\;\text{ in }\;\;\Omega\ , \qquad I_0(t) = 0 \;\;\text{ on }\;\;\partial\Omega\ . \label{z9} 
\end{equation}
Furthermore, we may argue as in the proof of \cite[Lemma~2.6]{ELW2} and use the time continuity of $\bar{u}$ in $W_q^2(I)$ to show that
\begin{equation}
I_0 \in C([0,T],W_{2,D}^1(\Omega))\ . \label{z10}
\end{equation}
We then infer from \eqref{z3} and \eqref{z9} that, for $s\in (-t,T-t)$,
\begin{equation*}
\mathcal{L}_{\bar{u}(t)} \left( \frac{\delta_s \Phi(t)}{s} - I_0(t) \right) = \frac{R_s(t)}{s} - R_0(t) \;\;\text{ in }\;\;\Omega\ , \qquad \frac{\delta_s \Phi(t)}{s} - I_0(t) = 0 \;\;\text{ on }\;\;\partial\Omega\ ,
\end{equation*}
and, using again \cite[Lemma~2.2]{ELW2}, we obtain
$$
\left\| \frac{\delta_s\Phi(t)}{s} - I_0(t) \right\|_{W_{2,D}^{1}(\Omega)} \le C \left\| \frac{R_s(t)}{s} - R_0(t) \right\|_{W_{2,D}^{-1}(\Omega)} \ .
$$
Recalling \eqref{z8}, we conclude that $\Phi$ is differentiable with respect to time in $W_{2,D}^1(\Omega)$ with derivative $I_0$. The latter being continuous by \eqref{z10}, we have thus established that 
\begin{equation}
\Phi \in C^1([0,T],W_{2,D}^1(\Omega)) \;\;\text{ with }\;\; \partial_t \Phi = I_0\ . \label{z11}
\end{equation}

\medskip

\noindent\textbf{Step~3: Time differentiability of  $\mathcal{E}_e(\bar{u})$}. Since $\phi(t,x,\eta) = \Phi(t,x,\eta)+\eta$ for $(t,x,\eta)\in [0,T]\times\Omega$, it follows from \eqref{z11} that $\phi \in C^1([0,T],W_{2,D}^1(\Omega))$ with $\partial_t \phi = \partial_t \Phi = I_0$. Thanks to this property, we readily deduce from \eqref{z4} that $\mathcal{E}_e(\bar{u}) \in C^1([0,T])$ with
\begin{align}
\frac{\rd }{\rd t}\mathcal{E}_e(\bar{u}) = & 2 \varepsilon^2 \int_\Omega \left( \partial_x \phi - \eta U \partial_\eta \phi \right) \left( \partial_x \partial_t \phi - \eta U \partial_\eta \partial_t \phi - \eta \partial_t U \partial_\eta \phi \right) (1+\bar{u})\ \rd (x,\eta) \nonumber \\
& + \varepsilon^2 \int_\Omega \left( \partial_x \phi - \eta U \partial_\eta \phi \right)^2 \partial_t \bar{u}\ \rd(x,\eta) - \int_\Omega \left( \partial_\eta \phi \right)^2 \frac{\partial_t \bar{u}}{(1+\bar{u})^2}\ \rd (x,\eta) \nonumber \\
& + 2 \int_\Omega \frac{\partial_\eta \phi \partial_\eta \partial_t \phi}{1+\bar{u}}\ \rd (x,\eta)\ . \label{z12}
\end{align}
Since $\partial_t \phi= \partial_t \Phi =0$ on $(0,T)\times\partial\Omega$, it follows from \eqref{230} and Green's formula that, for $t\in [0,T]$, 
\begin{align}
0 = & \int_\Omega (1+\bar{u}) \partial_t \phi \ \mathcal{L}_{\bar{u}} \phi\ \rd (x,\eta) \nonumber \\
= & - \varepsilon^2 \int_\Omega \left( \partial_x \bar{u} \partial_t \phi + (1+\bar{u}) \partial_x \partial_t \phi \right) \left( \partial_x \phi - \eta U \partial_\eta \phi \right)\ \rd (x,\eta) \nonumber \\
& + \varepsilon^2 \int_\Omega \eta (1+\bar{u}) U \partial_\eta \partial_t \phi \left( \partial_x \phi - \eta U \partial_\eta \phi \right) \rd (x,\eta) - \int_\Omega \partial_\eta \partial_t \phi 
\frac{\partial_\eta \phi}{1+\bar{u}}\ \rd (x,\eta) \nonumber \\
& + \varepsilon^2 \int_\Omega \partial_x \bar{u} \partial_t \phi \left( \partial_x \phi - \eta U \partial_\eta \phi \right) \rd (x,\eta) \nonumber \\
= & - \varepsilon^2 \int_\Omega (1+\bar{u}) \left( \partial_x \phi - \eta U \partial_\eta \phi \right) \left( \partial_x \partial_t \phi - \eta U \partial_\eta \partial_t \phi \right)\ \rd (x,\eta) \nonumber \\
& - \int_\Omega \partial_\eta \partial_t \phi \frac{\partial_\eta \phi}{1+\bar{u}}\ \rd (x,\eta) \ . \label{z13}
\end{align}
Combining \eqref{z12} and \eqref{z13}, we find
\begin{align*}
\frac{\rd }{\rd t}\mathcal{E}_e(\bar{u}) = & - 2 \varepsilon^2 \int_\Omega \left( \partial_x \phi - \eta U \partial_\eta \phi \right) \eta \partial_t U \partial_\eta \phi (1+\bar{u}) \ \rd (x,\eta) \\
& + \varepsilon^2 \int_\Omega \left( \partial_x \phi - \eta U \partial_\eta \phi \right)^2 \partial_t \bar{u}\ \rd(x,\eta) - \int_\Omega \frac{\left( \partial_\eta \phi \right)^2}{(1+\bar{u})^2} \partial_t \bar{u}\ \rd (x,\eta)\ .
\end{align*}
Coming back to the original variables $(x,z)$ and function $\psi$ and using the identity
$$
\partial_t U = \partial_t \partial_x \left( \ln{(1+\bar{u})} \right) = \partial_x \left( \frac{\partial_t \bar{u}}{1+\bar{u}} \right)\ ,
$$
we obtain with the help of Green's formula, the property $\partial_t \bar{u} \in W_{q,D}^1(I)$, and \eqref{psi}-\eqref{psibc} (with $\bar{u}$ instead of $u$)
\begin{align*}
\frac{\rd }{\rd t}\mathcal{E}_e(\bar{u}) = & - 2 \varepsilon^2 \int_{\Omega(\bar{u})} (1+z) \partial_x \psi \partial_z \psi \partial_t U\ \rd (x,z) + \int_{\Omega(\bar{u})} \left( \varepsilon^2 |\partial_x\psi|^2 - |\partial_z\psi|^2 \right) \frac{\partial_t \bar{u}}{1+\bar{u}}\ \rd (x,z) \\
=\ & 2 \varepsilon^2 \int_{\Omega(\bar{u})} (1+z) \left( \partial_x^2 \psi \partial_z \psi + \partial_x \psi \partial_x \partial_z \psi\right) \frac{\partial_t \bar{u}}{1+\bar{u}} \ \rd (x,z) \\
& + 2 \varepsilon^2 \int_{-1}^1 \partial_x \psi(\cdot,\bar{u}) \partial_z \psi(\cdot,\bar{u}) \partial_x \bar{u} \partial_t \bar{u}\ \rd x + \int_{\Omega(\bar{u})} \left( \varepsilon^2 |\partial_x\psi|^2 - |\partial_z\psi|^2 \right) \frac{\partial_t \bar{u}}{1+\bar{u}}\ \rd (x,z) \\
=\ & 2 \int_{\Omega(\bar{u})} (1+z) \left( - \partial_z^2 \psi \partial_z \psi + \varepsilon^2 \partial_x \psi \partial_x \partial_z \psi\right) \frac{\partial_t \bar{u}}{1+\bar{u}} \ \rd (x,z) \\
& - 2 \varepsilon^2 \int_{-1}^1 |\partial_x \bar{u}|^2 |\partial_z \psi(\cdot,\bar{u})|^2 \partial_t \bar{u}\ \rd x + \int_{\Omega(\bar{u})} \left( \varepsilon^2 |\partial_x\psi|^2 - |\partial_z\psi|^2 \right) \frac{\partial_t \bar{u}}{1+\bar{u}}\ \rd (x,z) \\
= & - \int_{-1}^1 |\partial_z \psi(\cdot,\bar{u})|^2 \partial_t \bar{u}\ \rd x + \varepsilon^2 \int_{-1}^1 |\partial_x \psi(\cdot,\bar{u})|^2\, \partial_t \bar{u} \ \rd x \\
& - 2 \varepsilon^2 \int_{-1}^1 |\partial_x \bar{u}|^2 |\partial_z \psi(\cdot,\bar{u})|^2 \partial_t \bar{u}\ \rd x \\
= & - \int_{-1}^1 g(\bar{u}) \partial_t \bar{u} \ \rd x\ .
\end{align*}
Integration with respect to time completes the proof.
\end{proof}

\section{Well-Posedness}\label{Sec3a} 

According to Proposition~\ref{L1} we may write \eqref{hyper1}-\eqref{psibc} as a semilinear evolution equation 
\bqn\label{CPP}
\gamma^2\dfrac{\rd^2}{\rd t^2} u+\dfrac{\rd}{\rd t} u+ A u= -\lambda g(u)\ ,\quad t>0\ ,\qquad u(0)=u^0\ ,\quad\gamma^2\dfrac{\rd}{\rd t}  u(0)= \gamma^2 u^1\ ,
\eqn
only involving the deflection $u$, where the operator $A\in\mathcal{L}\big(H_D^4(I),L_2(I))\big)$ is given by
$$
Av:= \beta\partial_x^4 v-\tau\partial_x^2 v\ ,\quad v\in H_D^4(I)\ .
$$
Once \eqref{CPP} is solved, the solution $\psi(t)=\psi_{u(t)}$ to \eqref{psi}-\eqref{psibc} is obtained from Proposition~\ref{L1} and the subsequent discussion. Note that the operator $-A$ is the generator of an analytic semigroup on $L_2(I)$ with an exponential decay, see \cite{AmannTeubner} or \cite[Theorem~7.2.7]{Pazy}.

\subsection{Parabolic Case: $\gamma=0$}

In that case, the equation \eqref{CPP} reduces to the parabolic semilinear Cauchy problem
\bqn\label{CP}
\dfrac{\rd}{\rd t} u+ A u= -\lambda g(u)\ ,\quad t>0\ ,\qquad u(0)=u^0\ .
\eqn
Since $-A$ generates an exponentially decaying analytic semigroup on $L_2(I)$, the global Lipschitz property of the function $g$ stated in Proposition~\ref{L1} ensures that we may prove exactly as in \cite[Theorem~1]{ELW1} the following existence result, for which we thus omit details:

\begin{prop}[{\bf Well-Posedness}]\label{Aglobal}
Let $\gamma=0$. Given  $4\xi\in (2,4]$, consider an initial value \mbox{$u^0\in  H_\B^{4\xi}(I)$} such that $u^0(x)>-1$ for $x\in I$. Then, the following hold:
\begin{itemize}

\item[(i)]  For each $\lambda>0$, there is a unique solution $(u,\psi)$ to  \eqref{hyper1}-\eqref{psibc} on the maximal interval of existence $[0,T_m)$ in the sense that
$$
u\in C^1\big((0,T_m),L_2(I)\big)\cap C\big((0,T_m), H_\B^4(I)\big) \cap C\big([0,T_m), H_\B^{4\xi}(I)\big)
$$
satisfies \eqref{hyper1}-\eqref{hyper3} together with
$$
u(t,x)>-1\ ,\quad (t,x)\in [0,T_m)\times I\ , 
$$ 
and $\psi(t)\in H^2\big(\Omega(u(t))\big)$  solves \eqref{psi}-\eqref{psibc} in $\Omega(u(t))$ for each $t\in [0,T_m)$. In addition, if $\xi=1$, then $u\in C^1\big([0,T_m),L_2(I)\big)$.

\item[(ii)] If, for each $T>0$, there is $\kappa(T)\in (0,1)$ such that 
$$
\|u(t)\|_{H^{4\xi}(I)}\le \frac{1}{\kappa(T)}\ ,\qquad u(t)\ge -1+\kappa(T)\ \text{ in } I
$$ 
for $t\in [0,T_m)\cap [0,T]$, then the solution exists globally in time, that is, $T_m=\infty$.

\item[(ii)] Given $\kappa\in (0,1)$, there is $\lambda_*(\kappa)>0$ such that the solution exists globally in time provided that $\lambda\in (0,\lambda_*(\kappa))$ and $u^0\ge -1+\kappa$ on $I$ with $\|u^0\|_{H_\B^{4\xi}(I)}\le 1/\kappa$. Moreover,  $u\in L_\infty(0,\infty,H_\B^{4\xi}(I))$ in this case with
$$
\inf_{(t,x)\in [0,\infty)\times I} u(t,x)>-1\ .
$$ 
\end{itemize}
\end{prop}

The statements~(i) and~(iii) of Theorem~\ref{Alin} readily follow from Proposition~\ref{Aglobal} with \mbox{$\xi=1$}. Notice that Proposition~\ref{Aglobal} is somewhat an extension of Theorem~\ref{Alin} as it requires weaker regularity on the initial condition. We shall prove the refined global existence criterion stated in part~(ii) of Theorem~\ref{Alin} in the next section, the starting point being Proposition~\ref{Aglobal}~(ii). 

\subsection{Hyperbolic Case: $\gamma>0$}

If $\gamma$ does not vanish, the equation \eqref{CPP} is hyperbolic, and we can no longer take advantage of the regularizing properties of the semigroup associated with the operator $-A$. In this case, we have to proceed in a different way outlined below. We shall prove the following refinement of Theorem~\ref{ThyperIntro}:

\begin{prop}\label{ThyperIntroduction}
Let $\gamma>0$ and $2\alpha\in (0,1/2)$. Consider an initial condition $(u^0,u^1)\in H_D^{4+2\alpha}(I)\times H_D^{2+2\alpha}(I)$  such that $u^0> -1$ in~$I$. Then the following hold:

\begin{itemize}
\item[(i)] For each $\lambda>0$, there is a unique solution $(u,\psi)$ to \eqref{hyper1}-\eqref{psibc} on a the maximal interval of existence $[0,T_m)$ in the sense that
$$
u\in C([0,T_m),H_D^{2+2\alpha}(I))\cap C^1([0,T_m),H_D^{2\alpha}(I))\ ,\quad \partial_t^k u\in L_1(0,T; H_D^{4+2\alpha-2k}(I))
$$
for $k=0,1,2$ and $T\in (0,T_m)$, and satisfies \eqref{hyper1}-\eqref{hyper3} together with
$$
u(t,x)>-1\ ,\quad (t,x)\in [0,T_m)\times I\ , 
$$ 
while $\psi(t)\in H^2\big(\Omega(u(t))\big)$  solves \eqref{psi}-\eqref{psibc} in $\Omega(u(t))$ for each $t\in [0,T_m)$. 

\item[(ii)] If $T_m<\infty$, then 
$$
\liminf_{t\rightarrow T_m}\ \min_{x\in[-1,1]}\{u(t,x)\} =-1
$$
or
$$
\limsup_{t\rightarrow T_m} \left(\|u(t)\|_{H_D^{2+2\alpha}(I)}+\|\partial_t u(t)\|_{H_D^{2\alpha}(I)}\right)=\infty\ .
$$

\item[(iii)] Given $\kappa\in (0,1)$, there are $\lambda(\kappa)>0$ and $N(\kappa)>0$ such that $T_m=\infty$ provided that $\lambda\le \lambda(\kappa)$, $u^0\ge -1+\kappa$ on $I$, and $$ \|(u^0,u^1)\|_{H_D^{4+2\alpha}(I)\times H_D^{2+2\alpha}(I)}\le N(\kappa)\ .$$
In this case,  $u\in L_\infty(0,\infty;H_D^{2+2\alpha}(I))$ with
$$
\inf_{(t,x)\in [0,\infty)\times I} u(t,x)>-1\ .
$$
\end{itemize}
\end{prop}

For the proof of this proposition, we simplify notation by setting $\gamma=1$. We first reformulate~\eqref{CPP} as a first-order Cauchy problem by using well-known results on cosine functions for which we refer to e.g. \cite[Section 5.5 $\&$ Section 5.6]{A04}: as previously observed, the self-adjoint operator $-A =-\beta \partial_x^4 + \tau \partial_x^2$ with domain $H_D^4(I)$ generates an analytic semigroup on $L_2(I)$ with spectrum contained in $[\mathrm{Re}\, z<0]$. Its inverse $A^{-1}$ is a compact linear operator on $L_2(I)$, and the square root of $A$ is well-defined. Noticing that $A$ is associated with the continuous coercive form
$$
\langle u,v\rangle =\int_{-1}^1 \big(\beta\partial_x^2 u\,\partial_x^2 v +\tau\partial_x u\,\partial_x v\big)\,\mathrm{d} x\ ,\quad u,v\in H_D^2(I)\ ,
$$
the domain of the square root of $A$ is (up to equivalent norms) equal to  $H_D^2(I)$. Consequently, the matrix operator
$$
\mathbb{A}:=\left(\begin{matrix} 0 & -1\\ A & 1\end{matrix}\right)
$$
with domain $D(\mathbb{A}):=H_D^4(I)\times H_D^2(I)$ generates a strongly continuous semigroup $e^{-t\mathbb{A}}$, $t\ge 0$, on the Hilbert space $\mathbb{H}:=H_D^2(I)\times L_2(I)$ (it actually generates a group $e^{-t\mathbb{A}}$, $t\in\R$). Moreover, owing to the damping term $\rd u/\rd t$ in \eqref{CPP}, the semigroup has exponential decay (see, e.g. \cite{Batkai,HZ88}), that is, there are $M\ge 1$ and $\omega>0$ such that
$$
\|e^{-t\mathbb{A}}\|_{\mathcal{L}(\mathbb{H})}\le Me^{-\omega t}\ ,\quad t\ge 0\ .
$$
 Writing ${\bf u}_0=(u^0,u^1)$, ${\bf u}=(u,\partial_tu)$, and 
\begin{equation}\label{f1}
f({\bf u})=\left(\begin{matrix} 0\\-g(u)\end{matrix}\right)\ ,
\end{equation}
we may  reformulate \eqref{CPP} as a hyperbolic semilinear Cauchy problem
\begin{equation}\label{CPhyp}
\dot{{\bf u}}+\mathbb{A} {\bf u}=\lambda f({\bf u})\ ,\quad t>0\ ,\qquad {\bf u}(0)={\bf u}_0
\end{equation}
in $\mathbb{H}$ with $\dot{{\bf u}}$ indicating the time derivative. In order to have a Lipschitz continuous semilinearity~$f$, Proposition~\ref{L1} dictates to shift \eqref{CPhyp} to an interpolation space of more regularity, e.g. to the (complex) interpolation space
$\mathbb{H}_\alpha:=[\mathbb{H},D(\mathbb{A})]_\alpha$  for some $\alpha\in (0,1)$. Indeed, 
we derive from \cite[Chapter V]{LQPP} that the $\mathbb{H}_\alpha$-realization $-\mathbb{A}_\alpha$ of $-\mathbb{A}$, given by
$$
\mathbb{A}_\alpha {\bf u}:= \mathbb{A} {\bf u}\ ,\quad {\bf u}\in D(\mathbb{A}_\alpha):=\left\{{\bf u}\in D(\mathbb{A})\,;\, \mathbb{A}{\bf u}\in \mathbb{H}_\alpha\right\}\ ,
$$
generates a strongly continuous semigroup on $\mathbb{H}_\alpha$ with exponential decay
\bqn\label{expdec}
\|e^{-t\mathbb{A}_\alpha}\|_{\mathcal{L}(\mathbb{H}_\alpha)}\le M_\alpha e^{-\omega t}\ ,\quad t\ge 0\ .
\eqn
Since  (up to equivalent norms, see e.g. \cite{GuidettiMathZ})
\bqn\label{Ha}
\mathbb{H}_\alpha= H_D^{2+2\alpha}(I)\times H_D^{2\alpha}(I)\ ,
\eqn
elliptic regularity theory readily shows that
$$
D(\mathbb{A}_\alpha)= H_D^{4+2\alpha}(I)\times H_D^{2+2\alpha}(I)\ .
$$
Fix now $2\alpha \in (0,1/2)$. Clearly, given $u^0\in H_D^{2+2\alpha}(I)$ with $u^0>-1$ on $I$, the continuous embedding of $H_D^{2+2\alpha}(I)$ in $C([-1,1])$ ensures that there is $\kappa\in (0,1)$ such that $u^0\in S_{(\alpha+1)/2}(\kappa)$, this set being defined in \eqref{setST}. Proposition~\ref{L1} entails that $f:\overline{S}_{(\alpha+1)/2}(\kappa)\times H_D^{2\alpha}(I)\rightarrow \mathbb{H}_\alpha$ is bounded and uniformly Lipschitz continuous. Noticing that $\overline{S}_{(\alpha+1)/2}(\kappa)\times H_D^{2\alpha}(I)$ is endowed with the same topology as $\mathbb{H}_\alpha$, a classical fixed point argument then yields: 

\begin{lem}\label{P1}
Let $2\alpha\in (0,1/2)$ and $\kappa\in (0,1)$. Then, for each ${\bf u}_0=(u^0,u^1)\in \mathbb{H}_\alpha$ with $u^0\in S_{(\alpha+1)/2}(\kappa)$, the Cauchy problem \eqref{CPhyp} has a unique mild solution ${\bf u}=(u,\partial_t u)\in C([0,T_m),\mathbb{H}_\alpha)$ for some maximal time of existence $T_m=T_m({\bf u}_0) \in (0,\infty]$. If $T_m<\infty$, then 
 \begin{equation}\label{blowup0}
\liminf_{t\rightarrow T_m}\ \min_{x\in[-1,1]}\{u(t,x)\}=-1 
\end{equation}
or
\begin{equation}\label{blowup}
\limsup_{t\rightarrow T_m} \|(u(t),\partial_t u(t))\|_{\mathbb{H}_\alpha}=\infty\ .
\end{equation} 
\end{lem}

The proof of Lemma~\ref{P1} is classical. Nevertheless, since a refinement of it is needed later to show Proposition~\ref{ThyperIntroduction}~(iii), it will be sketched below.

To obtain more regularity on the mild solution ${\bf u}$, let us consider an initial condition in the domain $D(\mathbb{A}_\alpha)$ of the generator $-\mathbb{A}_\alpha$, that is, let ${\bf u}_0=(u^0,u^1) \in H_D^{4+2\alpha}(I)\times H_D^{2+2\alpha}(I)$ with $u^0\in S_{(\alpha+1)/2}(\kappa)$. Then, since $f$ is Lipschitz continuous,
it follows as in the proof of \cite[Theorem~6.1.6]{Pazy}  that  ${\bf u}: [0,T_m)\rightarrow \mathbb{H}_\alpha$ is Lipschitz continuous and whence differentiable almost everywhere with respect to time. Consequently, we obtain (see also \cite[Corollary~4.2.11]{Pazy}):
 
\begin{cor}\label{C31}
Let $2\alpha\in (0,1/2)$ and $\kappa\in (0,1)$. If ${\bf u}_0 =(u^0,u^1)\in H_D^{4+2\alpha}(I)\times H_D^{2+2\alpha}(I)$ with $u^0 \in S_{(\alpha+1)/2}(\kappa)$, then the mild solution ${\bf u}$ to \eqref{CPhyp} is actually a strong solution. That is,~${\bf u}$ is differentiable almost everywhere in time with derivative  $\dot{{\bf u}}$ satisfying $\dot{{\bf u}}\in L_1(0,T;\mathbb{H}_\alpha)$ for each $T\in (0,T_m)$ and $$\dot{{\bf u}}(t)=-\mathbb{A}_\alpha {\bf u}(t)+\lambda f({\bf u}(t))$$ in $\mathbb{H}_\alpha$ for almost every $t\in [0, T_m)$.
\end{cor} 

Under the assumption of Corollary~\ref{C31} we deduce from ${\bf u}=(u,\partial_t u)$ that, for each $T\in (0,T_m)$,
$$
\partial_t^k u\in C([0, T_m), H_D^{2+2\alpha-2k}(I))\ ,\quad \partial_t^{k+1} u\in  L_1(0, T ; H_D^{2+2\alpha-2k}(I))\ ,
$$ for $k=0,1$ and
\begin{equation}\label{pp}
A u=-\partial_t^2 u-\partial_t u-\lambda g(u)\ .
\end{equation}
Since $g(u)\in C([0,T_m),H_D^{2\alpha}(I))$ by Proposition~\ref{L1}, the right hand side of \eqref{pp} is in $L_1(0, T ; L_2(I))$ for each $T\in (0,T_m)$ and so we derive 
$$
u\in L_1(0, T ; H_D^{4+2\alpha}(I))\ , \quad T\in (0,T_m)\ .
$$
Thus, we have shown parts~(i) and~(ii) of Proposition~\ref{ThyperIntroduction}, and it remains to prove the global existence statement~(iii) for small voltage values $\lambda$ and small initial data therein. To this end, the fixed point argument leading to Lemma~\ref{P1} has to be refined.

\begin{proof}[Proof of Proposition~\ref{ThyperIntroduction}~(iii)]
Let ${\bf u}_0=(u^0,u^1) \in D(\mathbb{A}_\alpha)$ with $u^0\in S_{(\alpha+1)/2}(\kappa)$ for some $\kappa\in (0,1)$. Given $T>0$ introduce the complete metric space
$$
\mathcal{V}_T:=\left\{{\bf u}=(u_1,u_2)\in C([0,T],\mathbb{H}_\alpha)\,;\, u_1(t)\in \overline{S}_{(\alpha+1)/2}(\kappa/2) \text{ for } t\in [0,T]\right\}
$$
and define
$$
\Lambda({\bf u})(t):=\big( \Lambda_1({\bf u}),\Lambda_2({\bf u})\big)(t):= e^{-t\mathbb{A}_\alpha} {\bf u}_0 +\lambda\int_0^t e^{-(t-s)\mathbb{A}_\alpha} f({\bf u}(s))\,\rd s
$$
for $t\in [0,T]$ and ${\bf u}\in\mathcal{V}_T$. Since 
$$
e^{-t\mathbb{A}_\alpha} {\bf u}_0-{\bf u}_0=-\int_0^t e^{-s\mathbb{A}_\alpha} \mathbb{A}_\alpha{\bf u}_0\,\rd s\ ,\quad t\ge 0\ ,
$$
we obtain from \eqref{expdec} and Proposition~\ref{L1} that
\begin{equation*}
\begin{split}
\|\Lambda({\bf u})(t) - {\bf u}_0\|_{\mathbb{H}_\alpha}&\le \| e^{-t\mathbb{A}_\alpha} {\bf u}_0-{\bf u}_0\|_{\mathbb{H}_\alpha} +\lambda \int_0^t \|e^{-(t-s)\mathbb{A}_\alpha} f({\bf u}(s))\|_{\mathbb{H}_\alpha}\,\rd s\\
&\le 
 \frac{M_\alpha}{\omega} \|{\bf u}_0\|_{D(\mathbb{A}_\alpha)} +\lambda \frac{M_\alpha}{\omega} \sup_{v\in S_{(\alpha+1)/2}(\kappa)} \| g(v)\|_{H_D^{2\alpha}(I)}
\end{split}
\end{equation*}
for $t\in [0,T]$ and ${\bf u}\in\mathcal{V}_T$. We then also note, that, if $c_\alpha$ denotes the norm of the embedding of $H_D^{2+2\alpha}(I)$ in $L_\infty(I)$,
\begin{equation*}
\begin{split}
\Lambda_1({\bf u})(t) &= u^0 +\Lambda_1({\bf u})(t) -u^0 \ge -1+\kappa -c_\alpha\| \Lambda_1({\bf u})(t) -u^0\|_{H_D^{2+2\alpha}(I)}\\
& \ge -1+\kappa -c_\alpha \| \Lambda({\bf u})(t) -{\bf u}_0\|_{\mathbb{H}_\alpha}
\end{split}
\end{equation*}
on $I$. From these estimates it is immediate that there are $\lambda(\kappa)>0$ and $N(\kappa)>0$ such that $\Lambda:\mathcal{V}_T\rightarrow \mathcal{V}_T$ defines a contraction {\it for each} $T>0$ provided that $\lambda\le \lambda(\kappa)$ and $ \|{\bf u}_0\|_{D(\mathbb{A}_\alpha)}\le N(\kappa)$. Consequently, the strong solution ${\bf u}=(u,\partial_t u)$ exists globally in time and $ u(t)\in \overline{S}_{(\alpha+1)/2}(\kappa/2)$ for $t\ge 0$.\\
\end{proof}

\subsection{Additional Properties}
Let $\gamma\ge 0$ and let $(u,\psi)$ be the solution \eqref{hyper1}-\eqref{psibc} provided by Proposition~\ref{Aglobal} if $\gamma=0$ or Proposition~\ref{ThyperIntroduction} if $\gamma>0$. We first observe an immediate consequence of the uniqueness results of these propositions and the invariance of the equations with respect to the symmetry $(x,z)\to (-x,z)$.

\begin{cor}\label{Asim} 
If $u^0=u^0(x)$ in Proposition~\ref{Aglobal} or if $u^0=u^0(x)$ and $u^1=u^1(x)$  in Proposition~\ref{ThyperIntroduction}
are even with respect to $x\in I$, then, {for all $t\in [0,T_m)$}, $u=u(t,x)$ and \mbox{$\psi=\psi(t,x,z)$} are even with respect to $x\in I$ as well.
\end{cor}

We next improve the regularity of $\psi$ from Proposition~\ref{L1} when $\gamma=0$.

\begin{prop}\label{pr.regpsi}
Let $\gamma=0$. Under the assumptions of Proposition~\ref{Aglobal} with $\xi \ge 3/4$, the second component $\psi(t)$ of the solution to \eqref{hyper1}-\eqref{psibc} at time $t\in [0,T_m)$ belongs to $W_p^3(\Omega(u(t))$ for all $p\in (1,2)$.
\end{prop}

\begin{proof}
Fix $t\in [0,T_m)$. Owing to the continuous embedding of $H^{s}(I)$ in $W_\infty^2(I)$ for $s>5/2$, the operator $\mathcal{L}_{u(t)}$ defined in \eqref{acdc} has Lipschitz continuous coefficients (in $x$) when written in divergence form \eqref{Ldiv}, and its principal part at the four corners of $\Omega$ is simply $\varepsilon^2 \partial_x^2 + \partial_\eta^2$ thanks to the clamped boundary conditions \eqref{hyper2}. We then infer from \cite[Theorem~5.2.7]{Grisvard} that $(x,\eta)\mapsto \phi(t,x,\eta) - \eta$ belongs to $W_q^2(\Omega)$ for all $q\in (2,\infty)$, the function $\phi(t)=\phi_{u(t)}$ being the solution to \eqref{23}-\eqref{24}. Setting 
$$
\Psi(t,x,z) := \psi(t,x,z) - \frac{1+z}{1+u(t,x)} = \phi_{u(t)}\left( x ,  \frac{1+z}{1+u(t,x)} \right) - \frac{1+z}{1+u(t,x)} 
$$
for $(x,z)\in \Omega(u(t))$, we conclude that 
\begin{equation}
\Psi(t), \psi(t) \in W_q^2(\Omega(u(t)) \;\;\text{ for all }\;\; q\in (2,\infty)\ . \label{reg1}
\end{equation} 
Next, let $p\in (1,2)$. It follows from \eqref{psi}-\eqref{psibc} that $\Psi(t)$ solves
\begin{align}
\varepsilon^2 \partial_x^2 \Psi(t) + \partial_z^2 \Psi(t)  & = - \varepsilon^2 \partial_x^2 \left( \frac{1+z}{1+u(t)}\right) \;\;\text{ in }\;\; \Omega(u(t))\ , \label{reg2a} \\
\Psi(t) & = 0 \;\;\text{ on }\;\; \partial\Omega(u(t))\ , \label{reg2b}
\end{align}
and the right hand side of \eqref{reg2a} belongs to $W_p^1(\Omega(u(t))$ by Proposition~\ref{Aglobal} since $4\xi\ge 3$. Owing to the clamped boundary conditions \eqref{hyper2} and the constraint $p<2$, we are in a position to apply \cite[Theorem~5.1.3.1]{Grisvard} and conclude that $\Psi(t)$ belongs to $W_p^3(\Omega(u(t))$. Clearly, $\psi(t)$ also belongs to $W_p^3(\Omega(u(t))$ thanks to the regularity of $u(t)$, and the proof is complete.
\end{proof}

\section{Energy identities}\label{SectEnergy} 

The aim of this section is to establish the energy equalities \eqref{rc2} and \eqref{rc2b}. Under the assumptions of Proposition~\ref{Aglobal} if $\gamma= 0$ or the assumptions of Proposition~\ref{ThyperIntroduction} if $\gamma>0$ let $(u,\psi)$ be the solution \eqref{hyper1}-\eqref{psibc}. Since $u$ merely belongs to $C^1\big((0,T_m),L_2(I)\big)$, we cannot apply directly Proposition~\ref{pr.z1} and thus have to invoke an approximation argument. To this end, let us introduce the Steklov averages defined by 
$$
u_\delta (t,x):=\frac{1}{\delta}\int_t^{t+\delta} u(s,x)\,\rd s\ ,\qquad t\in [0,T_m)\ ,\quad x\in I\ ,\quad \delta\in (0,T_m-t)\ .
$$
Fix $T\in (0,T_m)$ and let $\delta\in (0,T_m-T)$ in the following. Owing to Proposition~\ref{Aglobal} or Proposition~\ref{ThyperIntroduction}, the function $u_\delta$ belongs to $C^1([0,T],H_{D}^{2+2\nu}(I))$ for some $\nu>0$  with 
$$
\partial_t u_\delta(t) = \frac{u(t + \delta) - u(t)}{\delta}\ ,\quad t\in [0,T]\ .
$$ 
In addition, 
\begin{equation}
u_\delta \longrightarrow u \ \text{ in }\ C([0,T], H_{D}^{2+2\nu}(I))\ \text{ as }\ \delta\to 0 \ , \label{spirou}
\end{equation}
which together with Proposition~\ref{L1} entails that
\begin{equation}
g(u_\delta)\longrightarrow g(u)\ \text{ in }\ C([0,T],L_2(I))\ \text{ as }\ \delta\to 0 \ . \label{fantasio}
\end{equation}
Moreover, if $\gamma>0$, then $\partial_t u\in C([0,T],L_2(I))$ by Proposition~\ref{ThyperIntroduction} so that  $\partial_t u_\delta \longrightarrow \partial_t u$ in $C([0,T],L_2(I))$ as $\delta\to 0$ and thus
\begin{equation}\label{spip}
 \partial_t u_\delta \longrightarrow \partial_t u\ \text{ in }\ L_2((0,T)\times I)\ \text{ as }\ \delta\to 0 \ .
\end{equation}
If $\gamma=0$, then $g(u)$ belongs to $C([0,T],L_2(I))$ by Proposition~\ref{L1} and Proposition~\ref{Aglobal} and thus also to $L_2(0,T;L_2(I))$. Hence, the maximal regularity property of the operator $A$ (see \cite[III.Example~4.7.3 \& III.Theorem~4.10.8]{LQPP}) in \eqref{CP} implies that $u\in W_2^1(0,T;L_2(I))$, from which we deduce \eqref{spip} in this case as well. Now, due to Proposition~\ref{pr.z1} we have
$$
\mathcal{E}_e(u_\delta(t_2)) - \mathcal{E}_e(u_\delta(t_1)) = -  \int_{t_1}^{t_2} \int_{-1}^1 g(u_\delta(s)) \partial_t u_\delta(s)\ \rd x \rd s\ , \quad 0\le t_1 \le t_2 \le T\ ,
$$
and, writing the integrals $\mathcal{E}_e(u_\delta(t_k))$ in terms of $\phi_{u_\delta(t_k)}$ on $\Omega$, $k=1,2$, and using Proposition~\ref{L1} along with \eqref{spirou}, \eqref{fantasio}, and \eqref{spip}, we are in a position to pass to the limit as $\delta\to 0$ in this identity and conclude that
\begin{equation}\label{uui}
\mathcal{E}_e(u(t_2)) - \mathcal{E}_e(u(t_1)) = -  \int_{t_1}^{t_2} \int_{-1}^1 g(u(s)) \partial_t u(s)\ \rd x \rd s\ , \quad 0\le t_1 \le t_2 \le T\ .
\end{equation}
Next, according to the regularity of $u$, $\partial_t u$, and $\gamma^2\partial_t^2 u$, a classical argument shows that
\begin{equation*}
\begin{split}
\frac{\gamma^2}{2}&\|\partial_tu(t_2)\|_{L_2(I)}^2+\mathcal{E}_b(u(t_2))+\mathcal{E}_s(u(t_2))-
\frac{\gamma^2}{2}\|\partial_tu(t_1)\|_{L_2(I)}^2-\mathcal{E}_b(u(t_1))-\mathcal{E}_s(u(t_1))\\
& = \int_{t_1}^{t_2}\int_{-1}^1  \left(\gamma^2\partial_t^2 u(s) +\beta \partial_x^4 u(s)-\tau \partial_x^2 u(s)\right) \partial_t u(s)\ \rd x \rd s
\end{split}
\end{equation*}
for $0\le t_1 \le t_2 \le T$. Finally, multiplying \eqref{uui} by $-\lambda$, adding the resulting identity to the previous equation, and using \eqref{hyper1} give \eqref{rc2} and \eqref{rc2b}.

\section{A refined criterion for global existence}\label{Sec6} 

We now shall improve the global existence criteria stated in parts~(ii) of Propositions~\ref{Aglobal} and~\ref{ThyperIntroduction} for small $\gamma$ by showing that norm blowup cannot occur in finite time, whence touchdown of $u$ on the ground plate is the only possible finite time singularity. The proofs for both cases follow the same lines, but differ at certain steps. We thus first provide the proof of the parabolic case $\gamma=0$ in the next subsection and perform then the one for the hyperbolic case $\gamma>0$ (with $\gamma$ small) in the subsequent subsection.

\subsection{Parabolic Case: $\gamma=0$}\label{rcge1}

\begin{prop}\label{pr.rc1}
Let $\gamma=0$.  Given $4\xi\in (2,4]$ and an initial condition $u^0\in  H_\B^{4\xi}(I)$ such that $u^0(x)>-1$ for $x\in I$, let $(u,\psi)$ be the corresponding solution to \eqref{hyper1}-\eqref{psibc} defined on $[0,T_m)$. If there are $T_0>0$ and $\kappa_0\in (0,1)$ such that
\begin{equation}
u(t,x) \ge -1 + \kappa_0\ , \qquad x \in I\ , \ \ t\in [0,T_0]\cap [0,T_m)\ , \label{rc1}
\end{equation}
then $T_m\ge T_0$ and
$$
\|u(t)\|_{H^{4\xi}(I)}\le c(\kappa_0,T_0)\ ,\quad t\in [0,T_0]\ . 
$$
\end{prop}

The remainder of this section is devoted to the proof of Proposition~\ref{pr.rc1} and requires several auxiliary results. From now on, $(u,\psi)$ is the solution to \eqref{hyper1}-\eqref{psibc} satisfying \eqref{rc1}.

\medskip

We first show that a weighted $L_1$-norm of $u$ is controlled during time evolution. We recall that, according to \cite[Theorem~4.6]{LaurencotWalker_JAM}, the operator $\beta\partial_x^4 - \tau \partial_x^2$ supplemented with the clamped boundary conditions \eqref{hyper2} has a positive eigenvalue $\mu_1>0$ with a corresponding positive eigenfunction $\zeta_1\in H_{D}^4(I)$ satisfying $\|\zeta_1\|_{L_1(I)}=1$, see also \cite{Gr02,Ow97}.  

\begin{lem}\label{le.rc3}
For $t\in [0,T_m)$,
\begin{equation}
\int_{-1}^1 \zeta_1(x) |u(t,x)|\ \mathrm{d}x \le 2 + \int_{-1}^1 \zeta_1(x) u(t,x)\ \mathrm{d}x \le 2 + \left( \int_{-1}^1 \zeta_1(x) u^0(x)\ \mathrm{d}x \right) \e^{-\mu_1 t}\ . \label{rc20}
\end{equation}
\end{lem}

\begin{proof}
It readily follows from \eqref{hyper1}, the properties of $\zeta_1$, and the non-positivity of the right hand side of \eqref{hyper1} that, for $t\in [0,T_m)$, 
$$
\frac{\mathrm{d}}{\mathrm{d}t} \int_{-1}^1 \zeta_1(x) u(t,x)\ \mathrm{d}x + \mu_1 \int_{-1}^1 \zeta_1(x) u(t,x)\ \mathrm{d}x \le 0\ ,
$$
hence
\begin{equation}
\int_{-1}^1 \zeta_1(x) u(t,x)\ \mathrm{d}x \le \left( \int_{-1}^1 \zeta_1(x) u_0(x)\ \mathrm{d}x \right) e^{-\mu_1 t}\ . \label{rc21}
\end{equation}
We next observe that, since $u(t,x)>-1$ for $(t,x)\in [0,T_m)\times I$,
\begin{align*}
\int_{-1}^1 \zeta_1(x) |u(t,x)|\ \mathrm{d}x = & \int_{-1}^1 \zeta_1(x) u(t,x)\ \mathrm{d}x  + 2 \int_{-1}^1 \zeta_1(x) \left( -u(t,x) \right)_+\ \mathrm{d}x \\
\le & \int_{-1}^1 \zeta_1(x) u(t,x)\ \mathrm{d}x  + 2 \int_{-1}^1 \zeta_1(x) \ \mathrm{d}x \ ,
\end{align*}
which, together with \eqref{rc21}, gives \eqref{rc20}.
\end{proof}

In view of Lemma~\ref{le.rc3}, the following Poincar\'e-like inequality shall be useful later on, its proof being performed by a classical contradiction argument, which we omit here.
 
\begin{lem}\label{le.rc4}
Given $\delta>0$, there is $K(\delta)>0$ such that
\begin{equation}
\|\partial_x v \|_{L_2(I)}^2 \le \delta \|\partial_x^2 v \|_{L_2(I)}^2 + K(\delta) \left(\int_{-1}^1 \zeta_1(x) |v(x)|\ \mathrm{dx}\right)^2\ , \quad v\in H_{D}^2(I)\ . \label{rc22}
\end{equation}
\end{lem}

We next investigate the relationship between $\mathcal{E}_b+\mathcal{E}_s$ and $\mathcal{E}_e$ introduced in \eqref{TotalEnergy}-\eqref{ElecEnergy} and begin with the following upper bound for the latter.

\begin{lem}\label{le.rc5}
For $t\in [0,T_m)$,
\begin{equation}
\mathcal{E}_e(u(t)) \le \int_{-1}^1 \left( 1 + \varepsilon^2 |\partial_x u(t,x)|^2 \right)\ \frac{\mathrm{d}x}{1+u(t,x)}\ . \label{rc23}
\end{equation}
\end{lem}

\begin{proof}
Let $t\in [0,T_m)$. We multiply \eqref{psi} by $\psi(t,x,z) - (1+z)/(1+u(t,x))$, integrate over $\Omega(u(t))$, and use Green's formula. Owing to \eqref{psibc}, the boundary terms vanish and we obtain
\begin{align*}
\mathcal{E}_e(u(t)) = & \int_{\Omega(u(t))} \left[ - \varepsilon^2 \partial_x \psi(t,x,z) \frac{(1+z)}{(1+u(t,x))^2} \partial_x u(t,x) + \frac{\partial_z \psi(t,x,z)}{1+u(t,x)} \right]\ \mathrm{d} (x,z) \ .
\end{align*}
We then infer from Young's inequality that 
\begin{align*}
\mathcal{E}_e(u(t))& \le  \frac{1}{2} \mathcal{E}_e(u(t))  + \frac{1}{2} \int_{\Omega(u(t))} \left[ \varepsilon^2 \frac{(1+z)^2}{(1+u(t,x))^4} |\partial_x u(t,x)|^2 + \frac{1}{(1+u(t,x))^2} \right]\ \mathrm{d} (x,z) \ ,
\end{align*}
from which \eqref{rc23} readily follows.
\end{proof}

Up to now, we have not used the lower bound \eqref{rc1} on $u$. It comes into play in the next result.

\begin{lem}\label{le.rc6}
There is $C_1(\kappa_0)>0$ such that
\begin{equation}
\mathcal{E}(u(t)) \ge \frac{1}{2}  \big(\mathcal{E}_b(u(t))+\mathcal{E}_s(u(t))\big) - C_1(\kappa_0)\ , \quad t\in [0,T_0]\cap [0,T_m)\ . \label{rc24}
\end{equation} 
\end{lem}

\begin{proof}
It follows from \eqref{rc1}, Lemma~\ref{le.rc4}, and Lemma~\ref{le.rc5} that, for $\delta>0$, 
\begin{align*}
\mathcal{E}(u(t)) \ge\ & \mathcal{E}_b(u(t))+\mathcal{E}_s(u(t)) - \lambda \int_{-1}^1 \left( 1 + \varepsilon^2 |\partial_x u(t,x)|^2 \right)\ \frac{\mathrm{d}x}{1+u(t,x)} \\
\ge\ & \mathcal{E}_b(u(t))+\mathcal{E}_s(u(t)) - \frac{\lambda}{ \kappa_0} \left( 2 + \varepsilon^2 \|\partial_x u(t)\|_{L_2(I)}^2 \right) \\
\ge\ & \mathcal{E}_b(u(t))+\mathcal{E}_s(u(t)) - \frac{\lambda}{\kappa_0} \left[ 2 + \varepsilon^2 \delta \|\partial_x^2 u(t)\|_{L_2(I)}^2 + \varepsilon^2 K(\delta) \left(\int_{-1}^1 \zeta_1(x) |u(t,x)|\ \mathrm{d}x\right)^2 \right] \ .
\end{align*}
The lower bound \eqref{rc24} then follows from the above inequality with the choice $\delta = \beta \kappa_0 / (4\lambda\varepsilon^2)$ and \eqref{rc20}.
\end{proof}

The last auxiliary result is a control of the right hand side of \eqref{hyper1} involving only $\mathcal{E}_b+\mathcal{E}_s$ and the lower bound \eqref{rc1}.

\begin{lem}\label{le.rc7}
Given $\sigma\in [0,1/2)$, there is $C_2(\kappa_0,\sigma)>0$ such that, for $t\in [0,T_m)$,
\begin{equation}
\| g(u(t))\|_{H^\sigma(I)} \le C_2(\kappa_0,\sigma)\ \left( 1 + \|u(t)\|_{H^2(I)}^{44} \right)\ . \label{rc24b}
\end{equation}
\end{lem}

\begin{proof}
We set 
$$
U(t,x) := \frac{\partial_x u(t,x)}{1+u(t,x)}  \;\;\text{ and }\;\; \Phi(t,x,\eta) := \phi(t,x,\eta)-\eta
$$ 
for $(t,x,\eta)\in [0,T_m)\times\Omega$, the function $\phi$ and the variable $\eta$ being defined in \eqref{23}-\eqref{24} and \eqref{Tuu}, respectively. Then $\Phi(t)$ solves
\begin{align}
- \big( \mathcal{L}_{u(t)} \Phi(t) \big)(x,\eta) & = f(t,x,\eta) := \varepsilon^2 \eta \left[ U(t,x)^2 - \partial_x U(t,x) \right] \;\;\text{ in }\;\;\Omega\ , \label{rc25} \\
\Phi(t,x,\eta) & = 0 \;\;\text{ on }\;\;\partial\Omega\ .
 \label{rc26} 
\end{align}
From now on, the time $t$ plays no particular role anymore and is thus omitted in the notation. We multiply \eqref{rc25} by $\Phi$, integrate over $\Omega$, and use Green's formula to obtain
\begin{align*}
P^2 := \varepsilon^2 \| \partial_x\Phi - \eta\ U\ \partial_\eta\Phi \|_{L_2(\Omega)}^2 + \left\| \frac{\partial_\eta\Phi}{1+u} \right\| _{L_2(\Omega)}^2 =\ & \varepsilon^2 \int_\Omega U \Phi \left( \partial_x\Phi - \eta\ U\ \partial_\eta\Phi \right) \ \mathrm{d} (x,\eta) \\
& + \int_\Omega f \Phi\ \mathrm{d}(x,\eta)\ .
\end{align*}
Observing that 
\begin{align*}
\int_\Omega f \Phi\ \mathrm{d}(x,\eta) =\ & \varepsilon^2 \int_\Omega \eta U^2 \Phi \ \mathrm{d}(x,\eta) + \varepsilon^2 \int_\Omega \eta U \partial_x \Phi\ \mathrm{d}(x,\eta) \\
=\ & \varepsilon^2 \int_\Omega \eta U^2 \Phi\ \mathrm{d}(x,\eta) + \varepsilon^2 \int_\Omega \eta^2 U^2 \partial_\eta\Phi\ \mathrm{d} (x,\eta)\\
& + \varepsilon^2 \int_\Omega \eta U \left( \partial_x \Phi - \eta\ U\ \partial_\eta \Phi \right)\ \mathrm{d}(x,\eta) \\
=\ & - \varepsilon^2 \int_\Omega \eta U^2 \Phi\ \mathrm{d}(x,\eta) \\
& + \varepsilon^2 \int_\Omega \eta U \left( \partial_x \Phi - \eta\ U\ \partial_\eta \Phi \right)\ \mathrm{d}(x,\eta) \ ,
\end{align*}
we end up with
\begin{align*}
P^2 =\ & \varepsilon^2 \int_\Omega U \left( \Phi + \eta \right) \left( \partial_x\Phi - \eta\ U\ \partial_\eta\Phi \right) \ \mathrm{d} (x,\eta) \\
& + \varepsilon^2 \int_\Omega \eta^2 U^2  \ \mathrm{d} (x,\eta) - \varepsilon^2 \int_\Omega \eta U^2 \left( \Phi + \eta \right)\ \mathrm{d}(x,\eta)\ .
\end{align*}
Since $0\le \phi = \Phi+\eta \le 1$ by the comparison principle and $\eta\in (0,1)$, we infer from Young's inequality that  
\begin{align*}
P^2 \le\ & \frac{\varepsilon^2}{2} \left[ \| U\|_{L_2(I)}^2 + \| \partial_x\Phi - \eta\ U\ \partial_\eta\Phi \|_{L_2(\Omega)}^2 \right] + \varepsilon^2 \| U\|_{L_2(I)}^2\ ,
\end{align*}
hence
\begin{equation}
\varepsilon^2 \| \partial_x\Phi - \eta\ U\ \partial_\eta\Phi \|_{L_2(\Omega)}^2 + \left\| \frac{\partial_\eta\Phi}{1+u} \right\| _{L_2(\Omega)}^2 \le 3 \varepsilon^2 \| U\|_{L_2(I)}^2\ . \label{rc27}
\end{equation}

Next, we multiply \eqref{rc25} (in non-divergence form) by $\zeta := \partial_\eta^2 \Phi$ and integrate over $\Omega$. We proceed as in \cite[Lemma~11]{LaurencotWalker_ARMA}\footnote{There is a sign misprint in the proof of \cite[Lemma~11]{LaurencotWalker_ARMA}.} with the help of \cite[Lemma~4.3.1.2 $\&$~4.3.1.3]{Grisvard} to deduce, with the notation $\omega := \partial_x \partial_\eta \Phi$, that
\begin{equation*}
Q^2 := \varepsilon^2\ \| \omega - \eta\ U\ \zeta \|_{L_2(\Omega)}^2 + 
\left\| \frac{\zeta}{1+u} \right\|_{L_2(\Omega)}^2 = - \int_\Omega f\ \left( 1+ \partial_\eta\Phi \right)\ \partial_\eta^2 \Phi\ \mathrm{d} (x,\eta) \ . 
\end{equation*}
We next use Green's formula, the boundary conditions \eqref{rc26} with their consequence \mbox{$\partial_\eta \Phi(\pm 1,\eta)=0$}, and the definition of $f$ to find
\begin{align}
Q^2 =\ & \frac{\varepsilon^2}{2} \int_\Omega \left( U^2 - \partial_x U \right) \left( 2 \partial_\eta \Phi + (\partial_\eta \Phi)^2 \right)\ \mathrm{d} (x,\eta) \nonumber \\
& - \frac{\varepsilon^2}{2} \int_{-1}^1 \left( U^2 - \partial_x U \right) \left( 2 \partial_\eta \Phi(\cdot, 1) + (\partial_\eta \Phi(\cdot, 1))^2 \right)\ \mathrm{d} x \nonumber \\
=\ & \frac{\varepsilon^2}{2} \int_\Omega \left( U^2 - \partial_x U \right) (\partial_\eta \Phi)^2\ \mathrm{d} (x,\eta) \nonumber \\
& -\frac{\varepsilon^2}{2} \int_{-1}^1 \left( U^2 - \partial_x U \right) \left( 2 \partial_\eta \Phi(\cdot, 1) + (\partial_\eta \Phi(\cdot, 1))^2 \right)\ \mathrm{d} x \nonumber \\
=\ & \frac{\varepsilon^2}{2} \int_\Omega U^2 (\partial_\eta \Phi)^2\ \mathrm{d} (x,\eta) + \varepsilon^2 \int_\Omega U   \omega \partial_\eta \Phi\ \mathrm{d} (x,\eta) \nonumber \\
& -\frac{\varepsilon^2}{2} \int_{-1}^1 \left( U^2 - \partial_x U \right) \left( 2 \partial_\eta \Phi(\cdot, 1) + (\partial_\eta \Phi(\cdot, 1))^2 \right)\ \mathrm{d} x \label{rc28}\ .
\end{align}
Observing that
\begin{align*}
\int_\Omega U \omega \partial_\eta \Phi\ \mathrm{d} (x,\eta) = & \int_\Omega U  \left( \omega - \eta U \zeta \right) \partial_\eta \Phi\ \mathrm{d} (x,\eta) + \int_\Omega \eta U^2    \partial_\eta \Phi \partial_\eta^2\Phi\ \mathrm{d} (x,\eta) \\
= & \int_\Omega U  \left( \omega - \eta U \zeta \right) \partial_\eta \Phi\ \mathrm{d} (x,\eta) - \frac{1}{2} \int_\Omega U^2 (\partial_\eta\Phi)^2\ \mathrm{d} (x,\eta) \\
&  + \frac{1}{2} \int_{-1}^1 U^2 (\partial_\eta\Phi(\cdot, 1))^2\ \mathrm{d} x
\end{align*}
by Green's formula, we combine the above inequality with \eqref{rc28} and use Cauchy-Schwarz and Young's inequalities to obtain
\begin{align*}
Q^2 & \le  \varepsilon^2 \int_\Omega U  \left( \omega - \eta U \zeta \right) \partial_\eta \Phi\ \mathrm{d} (x,\eta) - \varepsilon^2 \int_{-1}^1 U^2 \partial_\eta \Phi(\cdot, 1)\ \mathrm{d} x \\
&\quad + \frac{\varepsilon^2}{2} \int_{-1}^1 \partial_x U \left( 2 \partial_\eta \Phi(\cdot, 1) + (\partial_\eta \Phi(\cdot, 1))^2 \right)\ \mathrm{d} x \\
&\le  \frac{\varepsilon^2}{2} \|\omega - \eta U \zeta \|_{L_2(\Omega)}^2 + \frac{\varepsilon^2}{2} \| \partial_x u \|_{L_\infty(I)}^2 \left\| \frac{\partial_\eta \Phi}{1+u} \right\|_{L_2(\Omega)}^2 + \varepsilon^2 \|U\|_{L_4(I)}^2 \|\partial_\eta \Phi(\cdot, 1)\|_{L_2(I)} \\
&\quad + \varepsilon^2 \|\partial_x U\|_{L_2(I)}\ \left( 1 + \| \partial_\eta \Phi(\cdot, 1)\|_{L_4(I)}^2 \right)\ .
\end{align*}
Owing to the continuous embedding of $H^2(I)$ in $W_\infty^1(I)$ and in $W_4^1(I)$, \eqref{rc1}, and \eqref{rc27}, we deduce that there is a constant $C(\kappa_0)>0$ such that
\begin{equation}
Q^2 \le C(\kappa_0) \left( 1 + \|u\|_{H^2(I)}^4 \right) \left( 1 + \| \partial_\eta \Phi(\cdot, 1)\|_{L_4(I)}^2 \right)\ . \label{rc29}
\end{equation}
Now, given $\vartheta\in (3/4,1)$, we infer from the continuous embedding of $H^{1/4}(I)$ in $L_4(I)$ and the continuity of the trace operator from $H^\vartheta(\Omega)$ in $H^{1/4}(\partial\Omega)$ (see \cite[Theorem~1.5.1.2]{Grisvard}) that
$$
Q^2 \le C(\kappa_0) \left( 1 + \|u\|_{H^2(I)}^4 \right) \left( 1 + \| \partial_\eta \Phi\|_{H^\vartheta(\Omega)}^2 \right)\ .
$$
A classical interpolation inequality, the continuous embedding of $H^2(I)$ in $W_\infty^1(I)$, \eqref{rc1}, and \eqref{rc27} give
\begin{align*}
Q^2 \le\ & C(\kappa_0) \left( 1 + \|u\|_{H^2(I)}^4 \right) \left( 1 + \| \partial_\eta \Phi\|_{L_2(\Omega)}^{2(1-\vartheta)} \| \partial_\eta \Phi\|_{H^1(\Omega)}^{2\vartheta} \right) \\
\le\ & C(\kappa_0) \left( 1 + \|u\|_{H^2(I)}^4 \right) \left[ 1 + \| \partial_\eta \Phi\|_{L_2(\Omega)}^2 + \| \partial_\eta \Phi\|_{L_2(\Omega)}^{2(1-\vartheta)} \left( \|\omega\|_{L_2(\Omega)}^{2\vartheta} + \| \zeta\|_{L_2(\Omega)}^{2\vartheta} \right) \right] \\
\le\ & C(\kappa_0) \left( 1 + \|u\|_{H^2(I)}^4 \right) \left( 1 + \|1+u\|_{L_\infty(I)}^2 \left\| \frac{\partial_\eta \Phi}{1+u} \right\|_{L_2(\Omega)}^2 \right) \\
& + C(\kappa_0) \left( 1 + \|u\|_{H^2(I)}^4 \right) \|1+u\|_{L_\infty(I)}^{2(1-\vartheta)} \left\| \frac{\partial_\eta \Phi}{1+u} \right\|_{L_2(\Omega)}^{2(1-\vartheta)} \left( \|\omega\|_{L_2(\Omega)}^{2\vartheta} + \| \zeta\|_{L_2(\Omega)}^{2\vartheta} \right) \\
\le\ & C(\kappa_0) \left( 1 + \|u\|_{H^2(I)}^8 \right) \left( 1 + \|\omega\|_{L_2(\Omega)}^{2\vartheta} + \| \zeta\|_{L_2(\Omega)}^{2\vartheta} \right) \\
\le\ & C(\kappa_0) \left( 1 + \|u\|_{H^2(I)}^{8+2\vartheta} \right) \left( 1 + \|\omega - \eta U \zeta \|_{L_2(\Omega)}^{2\vartheta} + \left\| \frac{\zeta}{1+u} \right\|_{L_2(\Omega)}^{2\vartheta} \right) \\
\le \ & C(\kappa_0) \left( 1 + \|u\|_{H^2(I)}^{8+2\vartheta} \right) \left( 1 + Q^{2\vartheta} \right)\ .
\end{align*}
Since $\vartheta\in (3/4,1)$, Young's inequality gives
\begin{equation}
Q^2 \le C(\kappa_0) \left( 1 + \|u\|_{H^2(I)}^{(8+2\vartheta)/(1-\vartheta)} \right)\ . \label{rc30}
\end{equation}
Using again the continuous embedding of $H^2(I)$ in $W_\infty^1(I)$, \eqref{rc1}, and \eqref{rc27}, together with \cite[Chapter~2, Theorem~5.4]{Necas67}, we find
\begin{align*}
\|\partial_\eta\Phi(\cdot, 1) \|_{H^{1/2}(I)}^2 &\le  C \|\partial_\eta\Phi\|_{H^1(\Omega)}^2 = C \left(\|\partial_\eta\Phi\|_{L_2(\Omega)}^2 + \|\omega\|_{L_2(\Omega)}^2 + \|\zeta\|_{L_2(\Omega)}^2\right) \\
&\le  C \|1+u\|_{L_\infty(I)}^2 \left\| \frac{\partial_\eta \Phi}{1+u} \right\|_{L_2(\Omega)}^2 + 2C \|\omega - \eta U \zeta \|_{L_2(\Omega)}^2 \\
&\quad + 2C \|\partial_x u\|_{L_\infty(I)}^2 \left\| \frac{\zeta}{1+u} \right\|_{L_2(\Omega)}^2 + C \|1+u\|_{L_\infty(I)}^2 \left\| \frac{\zeta}{1+u} \right\|_{L_2(\Omega)}^2 \\
&\le  C \left( 1 + \|u\|_{H^2(I)}^2 \right) Q^2\ .
\end{align*}
Combining this last inequality with \eqref{rc30} and the continuity of the pointwise multiplication 
$$
H^{1/2}(I) \cdot H^{1/2}(I) \hookrightarrow H^{\sigma}(I)\ , \quad \sigma\in [0,1/2)\ ,
$$
see \cite[Theorem~2.1 \&~Remark~4.2(d)]{AmannMultiplication}, leads us to 
\begin{equation}
\left\| \left( \partial_\eta\Phi(\cdot, 1) \right)^2 \right\|_{H^{\sigma}(I)} \le  C(\kappa_0,\sigma)\ \left( 1 + \|u\|_{H^2(I)}^{10/(1-\vartheta)} \right)\ , \quad \sigma\in [0,1/2) \ . \label{rc31}
\end{equation}
Finally, let  $\sigma \in [0,1/2)$. It follows from \eqref{rc1}, \eqref{rc31}, the continuous embedding of $H^2(I)$ in $W_\infty^1(I)$, and continuity of pointwise multiplication $H^1(I)\cdot H^\sigma(I)\hookrightarrow H^\sigma(I)$ that
\begin{align*}
\| g(u)\|_{H^\sigma(I)} &\le  C \, \left\| \frac{1 + \varepsilon^2 (\partial_x u)^2}{(1+u)^2} \right\|_{H^1(I)}  \left\| \left( \partial_\eta\Phi(\cdot, 1) \right)^2 \right\|_{H^\sigma(I)} \\
&\le  \frac{C}{\kappa_0^3} \left( 1 + \|u\|_{H^2(I)}^{3} \right) \left\| \left( \partial_\eta\Phi(\cdot, 1) \right)^2 \right\|_{H^{\sigma}(I)} \\
&\le  C(\kappa_0) \left( 1 + \|u\|_{H^2(I)}^{(13-3\vartheta)/(1-\vartheta)} \right)\ ,
\end{align*}
and thus \eqref{rc24b} after choosing $\vartheta\in (3/4,1)$ accordingly.
\end{proof}

\begin{proof}[Proof of Proposition~\ref{pr.rc1}]
We combine the energy identity \eqref{rc2} and \eqref{rc24} to obtain
$$
\frac{1}{2} \left[ \mathcal{E}_b(u(t)) + \mathcal{E}_s(u(t)) \right] - C_1(\kappa_0) \le \mathcal{E}(u(t))
\le \mathcal{E}(u^0)\ , \qquad t\in [0,T_0]\cap [0,T_m)\ ,
$$
hence, thanks to Poincar\'e's inequality,
$$
\|u(t)\|_{H^2(I)} \le C(\kappa_0)\ , \qquad t\in [0,T_0]\cap [0,T_m)\ .
$$
This last bound and \eqref{rc24b} then ensure that
$$
\|g(u(t))\|_{H^\sigma(I)} \le C(\kappa_0)\ , \qquad t\in [0,T_0]\cap [0,T_m)\ ,
$$
with $\sigma\in [0,1/2)$.
Now fix $T_1\in (0, T_m)\cap (0,T_0)$. Classical parabolic regularity results for \eqref{CP} entail that
\begin{equation}
\|u(t)\|_{H^4(I)} \le C(\kappa_0, T_1)\ , \qquad t\in [T_1,T_0]\cap [T_1,T_m)\ , \label{rc100}
\end{equation}
 which, together with the assumption \eqref{rc1} prevents the occurrence of a singularity in $[T_1,T_0]$. Consequently, $T_m\ge T_0$.
\end{proof}

Combining now Proposition~\ref{pr.rc1} and Proposition~\ref{Aglobal}~(ii) we obtain the following criterion for global existence. It also implies part~(ii) of Theorem~\ref{Alin} by taking $\xi=1$.

\begin{cor}\label{c111}
Let $\gamma=0$. Given $4\xi\in (2,4]$ and an initial condition $u^0\in  H_\B^{4\xi}(I)$ such that $u^0>-1$ on~$I$, let $(u,\psi)$ be the solution to \eqref{hyper1}-\eqref{psibc} on the maximal interval of existence $[0,T_m)$. If, for each $T>0$, there is $\kappa(T)\in (0,1)$ such that
$$ 
u(t)\ge -1+\kappa(T)\ \text{ on } I\ ,\quad t\in [0,T_m)\cap [0,T]\ ,
$$ 
then  $T_m=\infty$.
\end{cor}

\medskip

\subsection{Hyperbolic Case: $\gamma>0$} \label{rcge2}

We now prove the counterpart of Proposition~\ref{pr.rc1} in the hyperbolic case $\gamma>0$. For this, however, we require $\gamma$ to be sufficiently small, the reason for this additional constraint will become clear in the proof of Lemma~\ref{le.rc3.b}.

Let $\mu_1>0$ be the positive eigenvalue of the operator $\beta\partial_x^4 - \tau \partial_x^2$ with clamped boundary conditions and let $\zeta_1\in H_{D}^4(I)$ be the corresponding positive eigenfunction satisfying $\|\zeta_1\|_{L_1(I)}=~1$ already introduced at the beginning of Section~\ref{rcge1}.

\begin{prop}\label{pr.rc1.b}
Let $\gamma^2\in (0,1/4\mu_1]$. Given $2\alpha\in (0,1/2)$ and an initial condition $(u^0,u^1)\in H_D^{4+2\alpha}(I)\times H_D^{2+2\alpha}(I)$  such that $u^0> -1$ in~$I$, let $(u,\psi)$ be the corresponding solution to \eqref{hyper1}-\eqref{psibc} defined on $[0,T_m)$. If there are $T_0>0$ and $\kappa_0\in (0,1)$ such that
\begin{equation}
u(t,x) \ge -1 + \kappa_0\ , \qquad x \in I\ , \ \ t\in [0,T_0]\cap [0,T_m)\ , \label{rc1b}
\end{equation}
then $T_m\ge T_0$ and
$$
\|u(t)\|_{H^{2+2\alpha}(I)}+\|\partial_tu(t)\|_{H^{2\alpha}(I)}\le c(\kappa_0,T_0)\ ,\quad t\in [0,T_0]\ . 
$$
\end{prop}

For the remainder of this subsection, $(u,\psi)$ is the solution to \eqref{hyper1}-\eqref{psibc} satisfying \eqref{rc1b}. 

\medskip

We first need the analogue of Lemma~\ref{le.rc3}.

\begin{lem}\label{le.rc3.b}
There is a constant $c_0>0$, depending only on $(u^0,u^1)$ and $\gamma$, such that
\begin{equation}
\int_{-1}^1 \zeta_1(x) |u(t,x)|\ \mathrm{d}x \le c_0\ ,\quad t\in [0,T_m)\ . \label{rc20.b}
\end{equation}
\end{lem}

\begin{proof}
Setting for $t\in [0,T_m)$
$$
X(t):=\int_{-1}^1 \zeta_1(x) u(t,x) \mathrm{d}x \ ,\quad b(t):=\lambda\int_{-1}^1 g(u(t))\,\zeta_1\,\rd x\ge 0\ ,
$$
it follows from \eqref{hyper1} that $X$ solves the ordinary differential equation
$$
\gamma^2\frac{\rd^2 X}{\rd t^2}+\frac{\rd X}{\rd t} +\mu_1 X=-b(t)\ ,\quad t\in [0,T_m)\ .
$$
First suppose that $\gamma^2\in (0,1/4\mu_1)$ and put
$$
\sigma_{\pm 1}:=\frac{-1\pm\sqrt{1-4\gamma^2\mu_1}}{2\gamma^2}\ .
$$
Then $X$ is given by
$$
X(t)=a_1 e^{t\sigma_{1}}+a_{-1}e^{t\sigma_{-1}}-\frac{2}{\sqrt{1-4\gamma^2\mu_1}}\int_0^t b(s)\, e^{-(t-s)/2\gamma^2}\,\sinh \left(\frac{\sqrt{1-4\gamma^2\mu_1}}{2\gamma^2} (t-s)\right)\,\rd s
$$
for some $a_{\pm 1} \in\R$ depending only on $X(0)$ and $\rd X(0)/\rd t$, that is, on $(u^0,u^1)$. Consequently,
since $b\ge 0$ and $\sigma_{\pm 1} <0$,
$$
X(t)\le |a_{1}|+|a_{-1}|\ ,\quad t\in [0,T_m)\ .
$$
Similarly, if  $\gamma^2=1/4\mu_1$, then
$$
X(t)=(a_1+a_{-1}t) e^{t\sigma_{1}}-\frac{1}{\gamma^2}\int_0^t b(s)\, (t-s)\, e^{(t-s)\sigma_1}\,\rd s\ ,
$$
where $\sigma_1:=-1/2\gamma^2$ and, again, $a_{\pm 1} \in\R$ depend only on $(u^0,u^1)$. Therefore, since $b\ge 0$ and $\sigma_1 <0$,
$$
X(t)\le (|a_1|+|a_{-1}|t)\,e^{\sigma_1 t}\le c_0\ ,\quad t\in [0,T_m)\ .
$$
We thus have obtained an upper bound on $X(t)$, and we complete the proof as in Lemma~\ref{le.rc3}.
\end{proof}

Let us point out that when $\gamma^2>1/4\mu_1$, the representation formula for $X$ in the previous proof involves sine and cosine functions, and one thus cannot exploit the non-negativity of $b$ to deduce an upper bound for $X$.

\medskip

Clearly,  Lemma~\ref{le.rc5} and Lemma~\ref{le.rc7} are still valid for the hyperbolic case $\gamma>0$ as they only make use of the elliptic equation \eqref{psi}-\eqref{psibc} (in its transformed form for $\Phi$). Moreover, Lemma~\ref{le.rc6} remains valid as well due to Lemma~\ref{le.rc3} and we thus may tackle the proof of Proposition~\ref{pr.rc1.b}.

\begin{proof}[Proof of Proposition~\ref{pr.rc1.b}]
We combine the energy identity \eqref{rc2b} and \eqref{rc24} to obtain
$$
\frac{1}{2} \left[ \mathcal{E}_b(u(t)) + \mathcal{E}_s(u(t)) \right] - C_1(\kappa_0) \le \mathcal{E}(u(t))
\le \mathcal{E}(u^0)+\frac{\gamma^2}{2}\|u^1\|_{L_2(I)}^2\ , \qquad t\in [0,T_0]\cap [0,T_m)\ ,
$$
hence, thanks to Poincar\'e's inequality,
$$
\|u(t)\|_{H^2(I)} \le C(\kappa_0)\ , \qquad t\in [0,T_0]\cap [0,T_m)\ .
$$
This last bound and Lemma~\ref{le.rc7} then ensure that
$$
\|g(u(t))\|_{H^{2\alpha}(I)} \le C(\kappa_0)\ , \qquad t\in [0,T_0]\cap [0,T_m)\ ,
$$
with $2\alpha\in (0,1/2)$, whence
$$
\|f({\bf u}(t))\|_{\mathbb{H}_\alpha} \le C(\kappa_0)\ , \qquad t\in [0,T_0]\cap [0,T_m)\ ,
$$
with the notation \eqref{f1} and \eqref{Ha}. Consequently, \eqref{CPhyp} and \eqref{expdec} imply
\begin{equation}
\|{\bf u}(t)\|_{\mathbb{H}_\alpha} \le C(\kappa_0, T_0)\ , \qquad t\in [0,T_0]\cap [0,T_m)\ , \label{rc100.b}
\end{equation}
which, together with the assumption \eqref{rc1} prevents the occurrence of \eqref{blowup}, so $T_m\ge T_0$.
\end{proof}

The criterion for global existence stated in Theorem~\ref{ThyperIntro.b} is now a consequence of Proposition~\ref{ThyperIntroduction}~(ii) and Proposition~\ref{pr.rc1.b}. Its refined version reads:

\begin{cor}\label{c111c}
Let $\gamma^2\in (0,1/4\mu_1]$. Given $2\alpha\in (0,1/2)$ and an initial value $(u^0,u^1)\in H_D^{4+2\alpha}(I)\times H_D^{2+2\alpha}(I)$  such that $u^0> -1$ in~$I$, let $(u,\psi)$ be the solution to \eqref{hyper1}-\eqref{psibc} defined on the maximal interval of existence $[0,T_m)$. If, for each $T>0$, there is $\kappa(T)\in (0,1)$ such that
$$ 
u(t)\ge -1+\kappa(T)\ \text{ on } I\ ,\quad t\in [0,T_m)\cap [0,T]\ ,
$$ 
then  $T_m=\infty$.
\end{cor}

\section{Steady states}\label{Sec4}

\subsection{Existence of stable steady states: Proof of Theorem~\ref{TStable2}~(i)}\label{Sec4a} 

The precise statement of existence and asymptotic stability of steady states to \eqref{hyper1}-\eqref{psibc} for small values of $\lambda$ is the following:

\begin{prop}\label{pr.stst}
\begin{itemize}
\item[(i)] Let $\kappa\in (0,1)$. There are $\delta=\delta(\kappa)>0$ and an analytic function $$[\lambda\mapsto U_\lambda]:[0,\delta)\rightarrow H_{D}^4(I)$$ such that $(U_\lambda,\Psi_\lambda)$ is for each $\lambda\in (0,\delta)$ the unique steady state to
\eqref{hyper1}-\eqref{psibc} enjoying the following properties
$$
\|U_\lambda\|_{H^4(I)} \le 1/\kappa\ , \quad -1+\kappa \le U_\lambda < 0 \;\;\text{ and }\;\; \Psi_\lambda\in W_2^2(\Omega(U_\lambda))\ .
$$
Moreover, $U_\lambda$ is even for $\lambda\in (0,\delta)$ and $U_0=0$.

\item[(ii)] Let $\gamma=0$ and $\lambda\in (0,\delta)$. There are $\omega_0,r,R>0$ such that for each initial value $u^0\in H_{D}^4(I)$ with $\|u^0-U_\lambda\|_{H^4} <r$, the solution $(u,\psi)$ to \eqref{hyper1}-\eqref{psibc} exists globally in time and
$$
\|u(t)-U_\lambda\|_{H^4(I)}+\|\partial_t u(t)\|_{L_{2}(I)} \le R e^{-\omega_0 t} \|u^0-U_\lambda\|_{H^4(I)}\ ,\quad t\ge 0\ .
$$

\item[(iii)] Let $\gamma>0$, $2\alpha\in(0,1/2)$,  and $\lambda\in (0,\delta)$. There are $\omega_0,r,R>0$ such that for each initial value $(u^0,u^1)\in H_D^{2+2\alpha}(I)\times H_D^{2\alpha}(I)$  with $\|u^0-U_\lambda\|_{H^{2+2\alpha}} +\|u^1\|_{H^{2\alpha}}<r$, the solution $(u,\psi)$ to \eqref{hyper1}-\eqref{psibc} exists globally in time and
$$
\|u(t)-U_\lambda\|_{H^{2+2\alpha}(I)}+\|\partial_t u(t)\|_{H^{2\alpha}(I)} \le R e^{-\omega_0 t} \big(\|u^0-U_\lambda\|_{H^{2+2\alpha}(I)}+\|u^1\|_{H^{2\alpha}}\big)\ ,\quad t\ge 0\ .
$$
\end{itemize}
\end{prop}

\begin{proof} Since the operator $-A$ is the generator of an analytic semigroup on $L_2(I)$ with exponential decay, we may apply the implicit function theorem and the principle of linearized stability  as in \cite[Theorem~3]{ELW1}  to prove parts (i) and (ii), see also \cite[Proposition~4.1]{LW_QAM} for a complete proof. The negativity of $U_\lambda$ is a consequence of the non-negativity of $g$ and the comparison principle established in \cite[Theorem~1.1]{LaurencotWalker_JAM}. For part (iii) we recall that in \eqref{CPhyp}, the function $f:S_{(\alpha+1)/2}(\kappa)\times H_D^{2\alpha}(I)\rightarrow \mathbb{H}_\alpha$  is continuously differentiable and that $-\mathbb{A}_\alpha$ is the generator of a strongly continuous semigroup on $\mathbb{H}_\alpha$ with exponential decay. Thus, linearizing \eqref{CPhyp} around the steady state $(U_\lambda,0)$ to obtain
$$
\dot{{\bf v}}+\big(\mathbb{A}_\alpha-\lambda Df((U_\lambda,0))\big)  {\bf v}=\lambda F({\bf v})\ ,\quad t>0\ ,\qquad {\bf v}(0)={\bf v}_0
$$
for ${\bf v}={\bf u}-(U_\lambda,0)$ with 
$$
F({\bf v}) := f\big( (U_\lambda,0)+{\bf v} \big) - f\big( (U_\lambda,0) \big) - Df\big( (U_\lambda,0) \big) {\bf v} = o(\|{\bf v}\|_{\mathbb{H}_\alpha})\ \text{ as }\ {\bf v}\to 0 \ .
$$
Noticing that $-\mathbb{A}_\alpha+\lambda Df((U_\lambda,0))$  is again the generator of a strongly continuous semigroup on $\mathbb{H}_\alpha$ with exponential decay for $\lambda$ sufficiently small (see \cite[Theorem 3.1.1]{Pazy}), the principle of linearized stability yields part~(iii).
\end{proof}

\subsection{Non-existence of steady states: Proof of Theorem~\ref{TStable2}~(ii)}\label{Sec4b} 

Finally, we prove that no steady state exists for large values of $\lambda$. To do so, let $(u,\psi)$ be a steady state to \eqref{hyper1}-\eqref{psibc} with regularity $u\in H^4_{D}(I)$, $\psi\in H^2(\Omega(u))$ and satisfying  $u>-1$ on $I$. Set
\begin{equation}
\gamma_m(x) := \partial_z \psi(x,u(x))  \;\;\text{ and }\;\; G(x) := \left( 1 + \ve^2 (\partial_x u(x))^2 \right) \gamma_m(x)^2\ ,\quad x\in I\ . \label{ne5}
\end{equation} 
Then $u$ solves 
\begin{equation}
\beta \partial_x^4 u - \tau \partial_x^2 u = - \lambda G \ , \quad x\in I\ , \label{ne1} 
\end{equation}
with clamped boundary conditions \eqref{hyper2}, and we infer from the non-negativity of $G$ and \cite[Theorem~1.1]{LaurencotWalker_JAM} that
\begin{equation}
-1 < u(x) < 0\ , \quad x\in I\ . \label{ne6}
\end{equation}
Also, it follows from \eqref{psi}, \eqref{psibc}, and the comparison principle that 
\begin{equation}
0 \le \psi(x,z) \le 1\ , \quad (x,z)\in\Omega(u)\ . \label{ne6.2}
\end{equation}

The proof of Theorem~\ref{TStable2}~(ii) is performed by a nonlinear variant of the eigenfunction method and requires the existence of a positive eigenfunction for the linear operator on the left-hand side of \eqref{ne1}, a property which is enjoyed by the operator $\beta \partial_x^4 - \tau \partial_x^2$ in $H_D^4(I)$ as already pointed out. Again, let $\zeta_1$ be the positive eigenfunction in $H_{D}^4(I)$ of the operator $\beta\partial_x^4 - \tau \partial_x^2$ with clamped boundary conditions satisfying $\|\zeta_1\|_{L_1(I)}=1$ and associated to the positive eigenvalue $\mu_1>0$ \cite{Gr02,LaurencotWalker_JAM,Ow97}.

Let us now recall some connections between $\psi$ and $u$ established in \cite{ELW1}. We begin with an easy consequence of \eqref{psibc} and the Cauchy-Schwarz inequality (see \cite[Lemma~9]{ELW1}).

\begin{lem}\label{le.n1}
There holds
\begin{equation}
\int_{-1}^1 \frac{\zeta_1(x)}{1+u(x)}\ \mathrm{d}x \le  \int_{\Omega(u)} \zeta_1(x) |\partial_z \psi(x,z)|^2\ \mathrm{d} (x,z)\ . \label{ne12}
\end{equation}
\end{lem}

The next result can be proved as \cite[Lemma~10]{ELW1} and follows from \eqref{psi}-\eqref{psibc} after multiplying \eqref{psi} by $\zeta_1 \psi$ and using Green's formula. 

\begin{lem}\label{le.n2}
There holds
\begin{align}
  \int_{-1}^1 \zeta_1(x)\ &\left( 1 + \varepsilon^2\ |\partial_x u(x)|^2 \right)\ \gamma_m(x)\ \mathrm{d}x  \nonumber\\
= &\int_{\Omega(u)} \zeta_1(x)\ \left[ \varepsilon^2 |\partial_x \psi(x,z)|^2 + |\partial_z \psi(x,z)|^2 \right]\ \mathrm{d} (x,z) \nonumber\\ 
&  - \frac{\ve^2}{2} \int_{\Omega(u)} \psi(x,z)^2\,  \partial_x^2\zeta_1\ \mathrm{d} (x,z) + \frac{\varepsilon^2}{2}\ \int_{-1}^1 u(x)\, \partial_x^2\zeta_1\ \mathrm{d} x\ . \label{ne11}
\end{align}
\end{lem}

We next introduce the solution $\mathcal{U}\in H_D^6(I)$ to
\begin{equation}
- \partial_x^2 \mathcal{U} = u \;\;\text{ in }\;\; I\ , \quad \mathcal{U}(\pm 1)=0\ , \label{ne7}
\end{equation}
and deduce from \eqref{ne6}, \eqref{ne7}, and the comparison principle that
\begin{equation}
- \frac{1}{2} \le \frac{x^2-1}{2} \le \mathcal{U}(x) \le 0\ , \quad x\in I\ . \label{ne8}
\end{equation}
It readily follows from \eqref{ne8} that $-1\le \partial_x \mathcal{U}(-1) \le 0 \le \partial_x \mathcal{U}(1) \le 1$ which, together with \eqref{ne6} and \eqref{ne7}, guarantees that
\begin{equation}
|\partial_x \mathcal{U}(x)| \le 1\ , \quad x\in I\ . \label{ne8.1}
\end{equation}
Let $\alpha\in (0,1]$ to be determined later on. We multiply \eqref{ne1} by $(1+\alpha \mathcal{U}) \zeta_1$ and integrate over $I$. Using \eqref{hyper2}, \eqref{ne7}, and recalling that $G$ is defined in \eqref{ne5}, we obtain
\begin{align*}
\lambda \int_{-1}^1  (1+\alpha \mathcal{U}) \zeta_1 G\ \mathrm{d}x & =  \int_{-1}^1 (1+\alpha \mathcal{U}) \zeta_1 \left( - \beta \partial_x^4 u + \tau \partial_x^2 u \right)\ \mathrm{d}x \\
&=  - \int_{-1}^1 \left[ (1+\alpha \mathcal{U}) \partial_x\zeta_1 + \alpha \zeta_1 \partial_x \mathcal{U} \right] \left( - \beta \partial_x^3 u + \tau \partial_x u \right)\ \mathrm{d}x \\
&=  \int_{-1}^1 \left[ (1+\alpha \mathcal{U}) \partial_x^2\zeta_1 + 2\alpha \partial_x\zeta_1 \partial_x \mathcal{U} - \alpha \zeta_1 u \right] \left( - \beta \partial_x^2 u + \tau u \right)\ \mathrm{d}x \\
&=  \beta \int_{-1}^1 \left[ (1+\alpha \mathcal{U}) \partial_x^3\zeta_1 + 3\alpha \partial_x^2\zeta_1 \partial_x \mathcal{U} - 3 \alpha u \partial_x\zeta_1 - \alpha \zeta_1 \partial_x u \right] \partial_x u\ \mathrm{d}x \\
&\quad + \tau \int_{-1}^1 \left[ (1+\alpha \mathcal{U}) \partial_x^2\zeta_1 + 2\alpha \partial_x\zeta_1 \partial_x \mathcal{U} - \alpha \zeta_1 u \right] u \mathrm{d}x\ .
\end{align*}
Since
\begin{align*}
& \int_{-1}^1 \left[ (1+\alpha \mathcal{U}) \partial_x^3\zeta_1 + 3\alpha \partial_x^2\zeta_1 \partial_x \mathcal{U} - 3 \alpha u \partial_x\zeta_1 - \alpha \zeta_1 \partial_x u \right] \partial_x u\ \mathrm{d}x \\
= & - \int_{-1}^1 \left[ (1+\alpha \mathcal{U}) \partial_x^4 \zeta_1 + 4\alpha\partial_x^3\zeta_1 \partial_x \mathcal{U} - 3 \alpha u \partial_x^2\zeta_1 \right] u\ \mathrm{d}x + \frac{3\alpha}{2} \int_{-1}^1 u^2 \partial_x^2\zeta_1\ \mathrm{d}x - \alpha \int_{-1}^1 \zeta_1 |\partial_x u|^2\ \mathrm{d}x\ ,
\end{align*}
and since $\zeta_1$ is an eigenfunction of $\beta\partial_x^4 - \tau \partial_x^2$ in $H_D^4(I)$, 
we conclude
\begin{align}
\lambda \int_{-1}^1  (1+\alpha \mathcal{U}) \zeta_1 G\ \mathrm{d}x = - \mu_1 \int_{-1}^1 (1+\alpha \mathcal{U}) \zeta_1 u\ \mathrm{d}x - & \alpha\beta \int_{-1}^1 \zeta_1 |\partial_x u|^2\ \mathrm{d}x \nonumber \\
- \alpha \int_{-1}^1 \left[ 4 \beta \partial_x^3\zeta_1 - 2 \tau \partial_x\zeta_1 \right] u \partial_x \mathcal{U} \ \mathrm{d}x - & \alpha \int_{-1}^1 \left[ \tau \zeta_1 - \frac{9\beta}{2} \partial_x^2\zeta_1 \right] u^2\ \mathrm{d}x\ . \label{ne8.2}
\end{align}

At this point, it follows from \eqref{ne6} and \eqref{ne8} that
$$
 - \mu_1 \int_{-1}^1 (1+\alpha \mathcal{U}) \zeta_1 u\ \mathrm{d}x \le  - \mu_1 \int_{-1}^1  \zeta_1 u\ \mathrm{d}x \le  \mu_1 \int_{-1}^1  \zeta_1 \ \mathrm{d}x = \mu_1\ ,
$$
and from \eqref{ne6}, \eqref{ne8.1}, and the non-negativity of $\zeta_1$ that 
\begin{align*}
\alpha \left| \int_{-1}^1 \left[ 4 \beta \partial_x^3\zeta_1 - 2 \tau \partial_x\zeta_1 \right] u \partial_x \mathcal{U} \ \mathrm{d}x \right|& \le \alpha (4 \beta+2 \tau) \|\zeta_1\|_{H_{D}^4(I)}\ ,\\
- \alpha \int_{-1}^1 \left[ \tau \zeta_1 - \frac{9\beta}{2} \partial_x^2\zeta_1 \right] u^2\ \mathrm{d}x& \le  \frac{9\alpha\beta}{2} \|\zeta_1\|_{H_{D}^4(I)} \ .
\end{align*}
Inserting these bounds in \eqref{ne8.2} and observing that
$$
1 + \alpha \mathcal{U} \ge 1 - \frac{\alpha}{2} \ge \frac{1}{2}\ , \quad x\in I\ ,
$$
due to \eqref{ne8}, we end up with 
\begin{equation}
\frac{\lambda}{2} \int_{-1}^1 \zeta_1 G\ \mathrm{d}x \le \lambda \int_{-1}^1  (1+\alpha \mathcal{U}) \zeta_1 G\ \mathrm{d}x \le \mu_1 + \alpha K_1 - \alpha\beta \int_{-1}^1 \zeta_1 |\partial_x u|^2\ \mathrm{d}x \label{ne9}
\end{equation}
for some positive constant $K_1$ depending only on $\beta$ and $\tau$. Thanks to \eqref{ne9}, Lemma~\ref{le.n1}, and Lemma~\ref{le.n2}, we are in a position to argue as in the proof of \cite[Theorem~2~(ii)]{ELW1} to complete the proof of Theorem~\ref{TStable2}. More precisely, let $\delta>0$ be a positive number to be determined later on. We infer from Young's inequality that 
$$
\frac{\delta}{2} \int_{-1}^1 \zeta_1 G\ \mathrm{d}x \ge \int_{-1}^1 \zeta_1 \left( 1 + \ve^2 (\partial_x u)^2 \right) \gamma_m\ \mathrm{d}x - \frac{1}{2\delta} \int_{-1}^1 \zeta_1 \left( 1 + \ve^2 (\partial_x u)^2  \right)\ \mathrm{d}x\ ,
$$
whence
$$
\frac{\lambda}{2} \int_{-1}^1 \zeta_1 G\ \mathrm{d}x \ge \frac{\lambda}{\delta} \int_{-1}^1 \zeta_1 \left( 1 + \ve^2 (\partial_x u)^2 \right) \gamma_m\ \mathrm{d}x - \frac{\lambda}{2\delta^2} \left( 1 + \ve^2 \int_{-1}^1 \zeta_1 (\partial_x u)^2\ \mathrm{d}x \right)\ .
$$
Using \eqref{ne6}, \eqref{ne6.2}, \eqref{ne12} and \eqref{ne11}, we further obtain
\begin{align*}
\frac{\lambda}{2} \int_{-1}^1 \zeta_1 G\ \mathrm{d}x &\ge  \frac{\lambda}{\delta} \int_{-1}^1 \frac{\zeta_1}{1+u}\ \mathrm{d}x - \frac{\lambda \ve^2}{2\delta} \int_{\Omega(u)} \psi^2(x,z)\ \partial_x^2 \zeta_1(x)\ \mathrm{d}(x,z) \\
&\quad + \frac{\lambda \ve^2}{2\delta} \int_{-1}^1 u \partial_x^2 \zeta_1\ \mathrm{d}x - \frac{\lambda}{2\delta^2} \left( 1 + \ve^2 \int_{-1}^1 \zeta_1 (\partial_x u)^2\ \mathrm{d}x \right) \\
&\ge  \frac{\lambda}{\delta} - \frac{\lambda \ve^2}{2\delta} \| \partial_x^2 \zeta_1\|_{L_1(I)} - \frac{\lambda \ve^2}{2\delta} \| \partial_x^2 \zeta_1\|_{L_1(I)} - \frac{\lambda}{2\delta^2} \left( 1 + \ve^2 \int_{-1}^1 \zeta_1 (\partial_x u)^2\ \mathrm{d}x \right) \\
&\ge  \frac{\lambda}{\delta} \left( 1 - \ve^2 \| \partial_x^2 \zeta_1\|_{L_1(I)}  - \frac{1}{2\delta}\right) - \frac{\lambda \ve^2}{2\delta^2} \int_{-1}^1 \zeta_1 (\partial_x u)^2\ \mathrm{d}x \ .
\end{align*}
Combining \eqref{ne9} and the above inequality leads us to
$$
\mu_1 + \alpha K_1 \ge \frac{\lambda}{\delta} \left( 1 - \ve^2 \| \partial_x^2 \zeta_1\|_{L_1(I)}  - \frac{1}{2\delta}\right) + \left( \alpha \beta - \frac{\lambda \ve^2}{2\delta^2} \right) \int_{-1}^1 \zeta_1 (\partial_x u)^2\ \mathrm{d}x\ .
$$ 
We now choose $\delta = \ve \sqrt{\lambda/(2\alpha\beta)}$ and find
$$
\mu_1 + \alpha K_1 \ge \frac{\sqrt{2\alpha\beta\lambda}}{\ve} \left( 1 - \ve^2 \| \partial_x^2 \zeta_1\|_{L_1(I)} \right)  - \frac{\alpha\beta}{\ve^2}\ .
$$ 
Choosing finally $\alpha = \alpha_\varepsilon := \min\{1,\ve^2\}$, we end up with
\begin{equation}
\frac{\ve^2}{2\alpha_\ve \beta} \left( \mu_1 + \left( K_1 + \frac{\beta}{\ve^2} \right) \alpha_\ve \right)^2 \ge \left( 1 - \ve^2 \| \partial_x^2 \zeta_1\|_{L_1(I)} \right)^2 \lambda\ . \label{ne15}
\end{equation}
Setting $\ve_*:=\| \partial_x^2 \zeta_1\|_{L_1(I)}^{-1/2}>0$, Theorem~\ref{TStable2}~(ii) now readily follows from \eqref{ne15}.

\section*{Acknowledgments}

The work of Ph.L. was partially supported by the CIMI (Centre International de Math\'ematiques et d'Informatique) Excellence program and by the Deutscher Akademischer Austausch Dienst (DAAD) while enjoying the hospitality of the Institut f\"ur Angewandte Mathematik, Leibniz Universit\"at Hannover.



\end{document}